\documentclass[11pt]{amsart}

 \usepackage[nobysame]{amsrefs}

 \AtBeginDocument{\def\MR#1{}}

\usepackage[english]{babel} 
\usepackage[UKenglish]{datetime}

\usepackage{amscd}
\usepackage{amsfonts}
\usepackage{amsmath}
\usepackage{amssymb}
\usepackage{amsthm}
\usepackage{amsxtra}
\usepackage{cases}
\usepackage{chngcntr}
\usepackage{cite}
\usepackage{color}
\usepackage{graphicx}
\usepackage{latexsym}
\usepackage{bm}
\usepackage{bbold}
\usepackage{mathtools}
\usepackage{microtype}
\usepackage{qsymbols}
\usepackage[Symbolsmallscale]{upgreek} 
\usepackage[active]{srcltx}
\usepackage{tikz}
\usepackage{tikz-cd}
\usepackage{url}
\usepackage{verbatim} 

\usepackage{enumitem}
\usepackage{hyperref}
\usepackage{cleveref}

\usepackage[charter]{mathdesign}

\newcommand{\bbfont}{\mathbb}

\usepackage{ifthen}

\makeatletter
\newcommand{\tfs}[1]
{
\ifthenelse{\equal{\f@shape}{n}}{\ensuremath{\mathrm{#1}}}
        {\ifthenelse{\equal{\f@shape}{sc}}{\ensuremath{\mathrm{#1}}}
                {\ifthenelse{\equal{\f@shape}{it}}{\ensuremath{\mathit{#1}}}
                        {\ifthenelse{\equal{\f@shape}{sl}}{\ensuremath{\mathit{#1}}}{}
                        }
                }
        }
}
\makeatother

\makeatletter
\newcommand{\btfs}[1]
{
\ifthenelse{\equal{\f@shape}{n}}{\ensuremath{\mathrm{#1}}}
        {\ifthenelse{\equal{\f@shape}{sc}}{\ensuremath{\mathrm{#1}}}
                {\ifthenelse{\equal{\f@shape}{it}}{\ensuremath{\mathit{#1}}}
                        {\ifthenelse{\equal{\f@shape}{sl}}{\ensuremath{\mathit{#1}}}{}
                        }
                }
        }
}
\makeatother

\newcommand{\ep}{{\epsilon}}

\newcommand{\CC}{{\bbfont C}}

\newcommand{\NN}{{\bbfont N}}

\newcommand{\RR}{{\bbfont R}}
\newcommand{\TT}{{\bbfont T}}

\newcommand{\ZZ}{{\bbfont Z}}

\newcommand{\upd}{{\mathrm{d}}}

\newcommand{\norm}[1]{{\lVert #1 \rVert}}

\newcommand{\braces}[1]{{\{ #1\}}}

\newcommand{\lrnorm}[1]{{\left\lVert #1 \right\rVert}}
\newcommand{\lrangles}[1]{{\left\langle #1\right\rangle}}
\newcommand{\lrbraces}[1]{{\left\{ #1\right\}}}

\newcommand{\set}[1]{\braces{#1}}

\newcommand{\desset}[1]{\braces{#1}}
\newcommand{\lrdesset}[1]{\lrbraces{#1}}

\newcommand{\di}[1]{\,\upd #1}

\newcommand{\cont}{\tfs{C}}

\newcommand{\bounded}{\tfs{B}}

\newcommand{\Cstar}{\ensuremath{\tfs{C}^\ast}}
\newcommand{\Calgebra}{\Cstar\!-algebra}
\newcommand{\Calgebras}{\Cstar\!-algebras}

\theoremstyle{plain}

\newtheorem{theorem}{Theorem}[section]
\newtheorem{proposition}[theorem]{Proposition}
\newtheorem{lemma}[theorem]{Lemma}
\newtheorem{corollary}[theorem]{Corollary}

\newtheorem*{theorem*}{Theorem}
\newtheorem*{proposition*}{Proposition}
\newtheorem*{lemma*}{Lemma}
\newtheorem*{corollary*}{Corollary}
\newtheorem*{conjecture*}{Conjecture}
\newtheorem*{assumption*}{Assumption}
\newtheorem*{hypothesis*}{Hypothesis}
\newtheorem*{question*}{Question}
\newtheorem*{problem*}{Problem}
\newtheorem*{task*}{Task}
\newtheorem*{addendum*}{Addendum}
\newtheorem*{idea*}{Idea}
\newtheorem*{suggestion*}{Suggestion}
\newtheorem*{context*}{Context}
\newtheorem*{exercise*}{Exercise}

\theoremstyle{definition}

\newtheorem{definition}[theorem]{Definition}

\newtheorem{remark}[theorem]{Remark}

\newtheorem*{definition*}{Definition}
\newtheorem*{example*}{Example}
\newtheorem*{remark*}{Remark}

\setlist[enumerate,1]{label=\textup{(\arabic*)},ref=\arabic*}
\setlist[enumerate,2]{label=\textup{(\alph*)},ref=\arabic{enumi}.\alph*}
\setlist[enumerate,3]{label=\textup{(\roman*)},ref=\arabic{enumi}.\alph{enumii}.\roman*}
\setlist[enumerate,4]{label=\textup{(\Alph*)},ref=\arabic{enumi}.\alph{enumii}.\roman{enumiii}.\Alph*}

\crefname{theorem}{Theorem}{Theorems}
\crefname{proposition}{Proposition}{Propositions}
\crefname{lemma}{Lemma}{Lemmas}
\crefname{corollary}{Corollary}{Corollaries}
\crefname{conjecture}{Conjecture}{Conjectures}
\crefname{definition}{Definition}{Definitions}
\crefname{example}{Example}{Examples}
\crefname{remark}{Remark}{Remarks}
\crefname{assumption}{Assumption}{Assumptions}
\crefname{hypothesis}{Hypothesis}{Hypotheses}
\crefname{question}{Question}{Questions}
\crefname{problem}{Problem}{Problems}
\crefname{task}{Task}{Tasks}
\crefname{addendum}{Addendum}{Addenda}
\crefname{idea}{Idea}{Ideas}
\crefname{suggestion}{Suggestion}{Suggestions}
\crefname{context}{Context}{Contexts}
\crefname{section}{Section}{Sections}
\crefname{subsection}{Section}{Sections}

\crefname{equation}{equation}{equations}
\crefname{enumi}{part}{parts}
\crefname{enumii}{part}{parts}
\crefname{enumiii}{part}{parts}
\crefname{enumiv}{part}{parts}

\numberwithin{equation}{section}

\allowdisplaybreaks 

\newcommand{\e}{\epsilon}
\newcommand{\ee}{\e^\prime}

\newcommand{\pa}{^{\phantom\ast}}
\newcommand{\pki}{^{\!\!\!\!\phantom{k_{i}}}}

\newcommand{\zij}{z_{ij}}
\newcommand{\zzij}{{\overline z}_{ij}}

\newcommand{\gen}{s}
\newcommand{\iso}{s}
\newcommand{\uni}{s}
\newcommand{\Gen}{S}

\newcommand{\Hl}{H_{l,n-l}}

\newcommand{\stntupleGenzeroshort}[1]{\Gen_{\{\zij\}; 0,n}}

\newcommand{\strep}{\pi_{\mathrm s}}
\newcommand{\resstrep}{\rho}

\newcommand{\ssp}{L_l}
\newcommand{\ipssp}[1]{\lrangles{#1}_{\ssp}}
\newcommand{\ipstsp}[1]{\lrangles{#1}_{\phi}}

\newcommand{\UAtwopar}[2]{{\mathcal A}_{\{\zij\};#1,#2}}
\newcommand{\IUAtwopar}[2]{{\mathcal A}^0_{\{\zij\};#1,#2}}
\newcommand{\PAtwopar}[2]{{\mathcal P}_{\{\zij\};#1,#2}}
\newcommand{\IPAtwopar}[2]{{\mathcal P}^0_{\{\zij\};#1,#2}}

\newcommand{\UAl}{\UAtwopar{l}{n-l}}
\newcommand{\IUAl}{\IUAtwopar{l}{n-l}}
\newcommand{\PAl}{\PAtwopar{l}{n-l}}
\newcommand{\IPAl}{\IPAtwopar{l}{n-l}}

\newcommand{\z}{z}
\newcommand{\zz}{\overline{z}}

\begin{document}


\title[Universal $\boldsymbol{{\mathrm C^\ast}}$-algebras]{Universal $\boldsymbol{{\mathrm C^\ast}}$-algebras generated by doubly non-commuting isometries}

\author[Marcel de Jeu]{Marcel de Jeu} \address{Marcel de Jeu, Mathematical Institute, Leiden University, P.O.\ Box 9512, 2300 RA Leiden, the Netherlands; and Department of Mathematics and Applied Mathematics, University of Pretoria, Cor\-ner of Lynnwood Road and Roper Street, Hatfield 0083, Pretoria, South Africa }
\email{mdejeu@math.leidenuniv.nl}

\author[Alexey Kuzmin]{Alexey Kuzmin}
\address{Alexey Kuzmin, Chalmers University of Technology and University of Gothenburg,	Department of Mathematical Sciences, SE-412 96, Gothenburg, Sweden}
\email{vagnard.k@gmail.com}

\author[Paulo R.\ Pinto]{Paulo R.\ Pinto}
\address{Paulo R.\ Pinto, Centro de An\'alise Matem\'atica, Geometria e Sistemas Din\^amicos, Departamento de Matem\'atica, Instituto Superior T\'ecnico, Universidade de Lisboa,
Av.\ Rovisco Pais 1, 1049-001 Lisbon, Portugal}
\email{ppinto@math.tecnico.ulisboa.pt}





\begin{abstract}
We give an explicit injective representation of the universal $\mathrm{C}^\ast$-alge\-bra that is generated by doubly non-commuting isometries. This injectivity allows us to prove that such universal algebras embed naturally into each other and also, when combined with Rieffel's theory of deformation, to show that they are nuclear and to compute their K-theory.
\end{abstract}

\maketitle


\section{Introduction and overview}\label{sec:introduction_and_overview}


\noindent 
In this paper, we study the universal \Calgebras\ in the following definition.

\begin{definition}\label{def:universal_algebra}
	Suppose that $n\geq 1$ is an integer and that, for all $i\neq j$ with $1\leq i,j\leq n$, $\zij\in\TT$ are given such that $\zij=\zzij$. Take $l$ such that $0\leq l\leq n$. Then we let $\UAl$ denote the universal unital $\Cstar$-algebra that is generated by
	elements $\gen_1,\dotsc,\gen_n$ such that
	\begin{equation}\label{eq:relations_in_universal_algebra}
		\gen_i^\ast \gen_j\pa=\zzij\pa \gen_j\pa \gen_i^\ast
	\end{equation}
	and
	\begin{equation}\label{eq:automatic_relations_in_universal_algebra}
		\gen_i\pa \gen_j\pa=\zij\pa \gen_j \gen_i
	\end{equation}
	for all $i\neq j$ with $1\leq i,j\leq n$,
	\begin{equation}\label{eq:isometries_in_universal_algebra}
		\gen_i^\ast \gen_i\pa=1
	\end{equation}
	for $i=1,\dotsc,l$, and
	\begin{equation}\label{eq:unitaries_in_universal_algebra}
		\gen_i^\ast \gen_i\pa=\gen_i\pa \gen_i^\ast=1
	\end{equation}
	for $i=l+1,\dotsc,n$.
\end{definition}

In view of \cref{eq:relations_in_universal_algebra,eq:automatic_relations_in_universal_algebra,eq:isometries_in_universal_algebra}, we say that these algebras are generated by \emph{doubly non-commuting isometries}.\footnote{It is known (see \cite[p.2671]{jorgensen_proskurin_samoilenko:2005}) that \cref{eq:automatic_relations_in_universal_algebra} follows from \cref{eq:relations_in_universal_algebra} and the fact that the $s_i$ are isometries, but we have included it nevertheless because it \emph{is} relevant later on when studying the universal unital involutive algebra that underlies $\UAl$.}

In \cite{de_jeu_pinto:2020}, we investigated the unital representations of $\UAl$. After establishing a simultaneous Wold decomposition for the images of the $s_i$, these (irreducible) representations can be classified up to unitary equivalence in terms of the (irreducible) representations of the various non-commutative tori that correspond naturally to the data in \cref{def:universal_algebra}. In \cite{popescu:2020} and \cite{rakshit_sarkar_suryawanshi:2022}, similar results are obtained for two more general classes of universal \Calgebras.

It is not obvious that there exist unital representations of $\UAl$ other than the zero representation on the zero space. The Wold decomposition in \cite{de_jeu_pinto:2020}, however, helps one to find the following non-trivial example, which is taken from \cite[Section~4]{de_jeu_pinto:2020}. We write  $\NN_0\coloneqq\{0,1,2,\dotsc\}$.

Set
\begin{equation}\label{eq:hilbert_space_definition}
	\Hl\coloneqq\ell^2(\NN_0)^{\otimes l}\otimes\ell^2(\ZZ)^{\otimes(n-l)},
\end{equation}
which is to be read as  $\ell^2(\ZZ)^{\otimes n}$ if an $n$-tuple with only unitaries is to be defined (the case where $l=0$), and as $\ell^2(\NN_0)^{\otimes n}$ if an $n$-tuple with only pure isometries is to be defined (the case where $l=n$). We let $\{\ep_k: k\geq 0\}$ denote the canonical basis of $\ell^2(\NN_0)$, and let $\{\e_k: k\in\ZZ\}$ denote the canonical basis of $\ell^2(\ZZ)$, so that  $\{\ep_{k_1}\otimes\dotsm\otimes \ep_{k_l}\otimes \ep_{k_{l+1}}\otimes\dotsm\otimes \ep_{k_n} : k_1,\dotsc, k_l\geq 0; \, k_{l+1},\dotsc,k_n\in\ZZ\}$ is a basis of $\Hl$. For typographical reasons, we shall often write $\ep_{k_1,\dotsc,k_n}$ for $\ep_{k_1}\otimes\dotsc\otimes \ep_{k_n}$,  $\ep_{k_1,\dotsc,k_{i-1},k_{i},k_{i+1}\dotsc,k_n}$ for $\ep_{k_1}\otimes\dotsm\otimes \ep_{k_{i-1}}\otimes \ep_{k_i}\otimes \ep_{k_{i+1}}\otimes\dotsm\otimes \ep_{k_n}$, etc. 

For $i=1,\dotsc,n$, define $\Gen_i\in \bounded(\Hl)$ by setting

\begin{equation}\label{eq:general_isometries}
	\Gen_i\pki \e_{k_1,\dotsc,k_{i-1}, k_{i},k_{i+1},\dotsc,k_n}\pki \coloneqq \z_{i,1}^{k_1}\dotsm\z_{i,i-1}^{k_{i-1}} \e_{k_1,\dotsc,k_{i-1},k_i+1,k_{i+1},\dotsc, k_n}\pki
\end{equation}
for all $k_1,\dotsc,k_l\geq 0$ and all $k_{l+1},\dotsc,k_n\in\ZZ$, where empty products that occur are to be read as 1. It is then straightforward to verify (see \cite[Section~4]{de_jeu_pinto:2020} for details) that the $\Gen_i$ are isometries on $\Hl$ satisfying \cref{eq:relations_in_universal_algebra,eq:automatic_relations_in_universal_algebra}, that $\Gen_1,\dotsc,\Gen_l$ are pure isometries, and that $\Gen_{l+1},\dotsc,\Gen_n$ are unitaries. The resulting unique unital representation
\[
\strep:\UAl\to\bounded (\Hl)
\]
such that
\[
\strep(\gen_i)=\Gen_i
\]
for $i=1,\dotsc,n$ is called the \emph{standard representation of $\UAl$}.

\begin{remark}
	For use in the remainder of this paper, we record from \cite[Section~4]{de_jeu_pinto:2020} that, for  $k_1,\dotsc,k_{l}\geq 0$ and $k_{l+1},\dotsc,k_n\in\ZZ$,
	\begin{equation}\label{eq:adjoint_of_pure_isometry}
		\Gen_i^\ast \e_{k_1,\dotsc,k_n} =
		\begin{cases}
			\zz_{i,1}^{k_1}\dotsm\zz_{i,i-1}^{k_{i-1}} \e_{k_1, \dotsc, k_{i-1}, k_i-1, k_{i+1} ,\dotsc, k_n}\pki & \textup{if } k_i\geq 1,\\
			0& \textup{if } k_i=0,
		\end{cases}
	\end{equation}
	for $i=1,\dotsc,l$ and that
	\begin{equation}\label{eq:adjoint_of_unitary}
		\Gen_i^\ast \e_{k_1,\dotsc,k_n}=
		\zz_{i,1}^{k_1}\dotsm\zz_{i,i-1}^{k_{i-1}} \e_{k_1, \dotsc, k_{i-1}, k_i-1, k_{i+1} ,\dotsc, k_n}\pki
	\end{equation}
	for $i=l+1,\dotsc,n$.
\end{remark}

The main result of the present paper is the injectivity of the standard representation $\strep$ of $\UAl$ (see \cref{res:standard_representation_is_injective}), as announced in \cite[Remark~4,7]{de_jeu_pinto:2020}. For this, we first construct a faithful unital completely positive map $\phi$ from $\UAl$ into the bounded operators on a Hilbert space. Stinespring's theorem then yields a representation $\pi_\phi$ of $\UAl$ on a larger Hilbert space, which we know to be injective as $\phi$ is faithful. A closer analysis of $\pi_\phi$ shows that it is essentially an inflation of the standard representation $\strep$, so that  $\strep$ is also injective.

	When $l=0$, $\UAl$ is a non-commutative torus. For this case, it is mentioned without further details in \cite[pp.3-4]{rieffel:1990} that the injectivity of the standard representation is a consequence of a certain GNS-construction. Our method in the general case extends the one sketched in \cite[pp.3-4]{rieffel:1990}. With the GNS construction now being replaced with the Stinespring construction, which is necessitated by the fact that not all generating isometries need be unitaries, matters become quite a bit more involved.

	\begin{remark}
		In \cite[p.6]{bhatt_saurabh_UNPUBLISHED:2023}, a representation $\Psi^{n-l,l}$ of $\UAl$ is defined which is easily seen to be unitarily equivalent to our standard representation. Using the fact that the non-commutative tori are generically simple, it is then established that  $\Psi^{n-l,l}$ is injective for generic values of the $z_{ij}$; see \cite[Theorem~ 2.8]{bhatt_saurabh_UNPUBLISHED:2023}. In \cite[Remark~2.9]{bhatt_saurabh_UNPUBLISHED:2023}, it is stated without further comments that this is, in fact, always the case. Our proof of the injectivity of the standard representation, which is valid for \emph{all} values of the $z_{ij}$, is fundamentally different from the one in \cite{bhatt_saurabh_UNPUBLISHED:2023} for generic $z_{ij}$,  and treats all values on an equal footing.
	\end{remark}

\begin{remark}\label{rem:mnemonic}
	The definitions in \cref{eq:general_isometries} are most easily remembered by using the following mnemonic. Replace $\e_{k_1,\dotsc,k_{i-1},k_{i},k_{i+1},\dotsc,k_n}$ with the monomial $\Gen_1^{k_1}\dotsm \Gen_{i-1}^{k_{i-1}} \Gen_{i}^{k_i} \Gen_{i+1}^{k_{i+1}}\dotsm \Gen_n^{k_n}$. Using \cref{eq:automatic_relations_in_universal_algebra}, `bring the acting initial factor $\Gen_i$ to its proper spot' by writing
	\[
	\bm {\Gen_i}\cdot \Gen_1^{k_1}\!\dotsm \Gen_{i-1}^{k_{i-1}} \bm{ \Gen_{i}^{k_i}} \Gen_{i+1}^{k_{i+1}}\dotsm \Gen_n^{k_n}\!=\!z_{i,1}^{k_1}\dotsm\z_{i,i-1}^{k_{i-1}}\Gen_1^{k_1}\dotsm \Gen_{i-1}^{k_{i-1}} \bm {\Gen_{i}^{k_i+1}} \Gen_{i+1}^{k_{i+1}}\dotsm \Gen_n^{k_n}
	\]
	and then replace
	$\Gen_1^{k_1}\dotsm \Gen_{i-1}^{k_{i-1}} \Gen_{i}^{k_i+1} \Gen_{i+1}^{k_{i+1}}\dotsm \Gen_n^{k_n}$ with $\e_{k_1,\dotsc,k_{i-1},k_{i}+1,k_{i+1},\dotsc,k_n}$ again.
	This observation is essential to the proof that the standard representation is injective; see the proof of \cref{res:action_equivalent_to_standard_representation}. Such an `action by reordering' will also be instrumental in \cref{sec:embeddings}.
\end{remark}

\medskip

\noindent This paper is organised as follows.

In \cref{sec:the_standard_representation_is_injective}, we show in a number of steps that the standard representation is injective.

One can select a subset of the generators of $\UAl$ and retain the corresponding structure constants $z_{ij}$.  This yields natural unital $\ast$-homomorphisms from `smaller' universal \Calgebras\ into `larger' ones. In \cref{sec:embeddings}, it is shown that `smaller' and `larger' are, indeed, correct to speak of: these natural unital $\ast$-homomorphisms are injective. The proof of this fact uses the injectivity of the standard representation.

In \cref{sec:deformation}, Rieffel's theory of deformation is applied. Using the fact that the standard representations are injective, it is shown that, for fixed $n$ and $l$, the algebras $\UAl$ are Rieffel deformations of each other for all choices of the $z_{ij}$. Taking the $z_{ij}$ equal to 1, it is then immediate that $\UAl$ is nuclear, a fact which is also established in \cite[Theorem~ 6.2]{rakshit_sarkar_suryawanshi:2022} for a larger class of algebras using a theorem of Rosenberg's (see \cite[Theorem~3]{rosenberg:1977}). Since K-theory is stable under Rieffel deformation, it is then straightforward to compute the K-groups of $\UAl$. Using more elaborate methods, these are also determined in \cite{bhatt_saurabh_UNPUBLISHED:2023}.

\section{Injectivity of the standard representation of $\UAl$}\label{sec:the_standard_representation_is_injective}


\noindent
In this section, we show that the standard representation $\strep$ of $\UAl$ on $\Hl$ is injective. As explained in \cref{sec:introduction_and_overview}, this is done by finding a faithful unital completely positive map $\phi$ from $\UAl$ into the bounded operators on a Hilbert space, and showing that the injective corresponding Stinespring representation is essentially an inflation of the standard representation $\strep$ of $\UAl$. Hence $\strep$ must itself be injective.

This programme is completed in a number of steps.

\subsection{The involutive algebra underlying $\UAl$}\label{subsec:underlying_involutive_subalgebra}

We let $\IUAl$ denote the universal unital involutive algebra generated by elements $\gen_1,\dots,\gen_n$ satisfying the relations in  \cref{eq:relations_in_universal_algebra,eq:automatic_relations_in_universal_algebra,eq:isometries_in_universal_algebra,eq:unitaries_in_universal_algebra}.
For $x\in\IUAl$, we set
\[
\norm{x}\coloneqq\sup\lrdesset{\norm{\pi(x)} : \pi\text{ is a unital $\ast$-representation of }\IUAl}.
\]
This supremum is finite since the $\pi(\gen_i)$ are always isometries, and $\norm{\,\cdot\,}$ is then a \Cstar-seminorm on $\IUAl$. If we let $I$ denote its kernel, then $\UAl$ is the completion of $\IUAl/I$ in the \Cstar-norm that $\norm{\,\cdot\,}$ induces on this quotient; see, e.g., \cite[II.8.3.1]{blackadar_OPERATOR_ALGEBRAS:2006} for the construction of universal \Cstar-algebras.

The first step in our programme is to show that $\norm{\,\cdot\,}$ is actually a norm on $\IUAl$, so that $\IUAl$ can be identified with a dense subalgebra of $\UAl$. For this, we again take $\Hl$ and $\Gen_1,\dotsc,\Gen_n$ as in \cref{eq:hilbert_space_definition,eq:general_isometries}, respectively, and we define\textemdash in an anticipating terminology and notation\textemdash the standard representation
\[
\strep:\IUAl\to \bounded(\Hl)
\]
of $\IUAl$ on $\Hl$ by requiring that $\strep(\gen_i)=\Gen_i$ for $i=1,\dotsc,n$. We shall show that the standard representation of $\IUAl$ is injective, which shows that $\norm{\,\cdot\,}$ is a norm.
We shall prove at the same time that a certain canonical spanning set for $\IUAl$ is, in fact, a basis of $\IUAl$

For this, we start by collecting a few algebraic results which will be used throughout the paper. The proofs are elementary and left to the reader.

\begin{lemma}\label{res:algebraic_results} Suppose that $n\geq 1$ is an integer and that, for all $i\neq j$ with $1\leq i,j\leq n$, $\zij\in\TT$ are given such that $\zij=\zzij$.
	Let $A$ be a unital involutive algebra, and suppose that $\gen_1,\dotsc,\gen_n$ are elements of $A$ satisfying the relations in  \cref{eq:relations_in_universal_algebra,eq:automatic_relations_in_universal_algebra}, and which are such that $\gen_i^\ast \gen_i=1$ for $i=1,\dotsc,n$.
	Then:
	\begin{enumerate}
		\item $\gen_j^\ast \gen_i\pa=\zij\pa \gen_i\pa \gen_j^\ast$ for all $i\neq j$ with $1\leq i,j\leq n$;
		\item $\gen_j^\ast \gen_i^\ast=\zz_{ij}\pa \gen_i^\ast \gen_j^\ast$ for all $i\neq j$ with $1\leq i,j\leq n$;
		\item $\gen_i^{\ast k}\gen_i^k=1$ for all $i$ with $1\leq i \leq n$ and all $k\geq 0$;
		\item $\gen_i^k \gen_i^{\ast k}$ is a selfadjoint idempotent for all $i$ with $1\leq i \leq n$ and all $k\geq 0$.
		\item for all $i$ with $1\leq i\leq n$ and for all $k$ and $l$ such that $0\leq k\leq l$,
		\[
		\left(\gen_i^k \gen_i^{\ast k}\right)\left(\gen_i^l \gen_i^{\ast l}\right)=\left(\gen_i^l \gen_i^{\ast l}\right)\left(\gen_i^k \gen_i^{\ast k}\right)=\gen_i^l \gen_i^{\ast l};
		\]
		\item $\gen_i\pa$ and $\gen_i^\ast$ commute with $\gen_j^k \gen_j^{\ast k}$ for all $i\neq j$ with $1\leq i,j\leq n$ and all $k\geq 0$.
	
	\end{enumerate}
\end{lemma}

We can now describe a spanning set of $\IUAl$ that will turn out to be a basis. First of all, as a consequence of \cref{eq:relations_in_universal_algebra,eq:automatic_relations_in_universal_algebra} and \cref{res:algebraic_results}, if $i\neq j$, then each of the elements $\gen_i\pa$ and $\gen_i^\ast$ of $\IUAl$ commutes with each of $\gen_j\pa$ and $\gen_j^\ast$ up to a multiplicative unimodular constant. This enables one to fix an ordering for the factors in an arbitrary word in the $\gen_i$ and the $\gen_i^\ast$. Up to a unimodular constant, such a word in $\IUAl$ is equal to a product in $\gen_1$ and $\gen_1^\ast$, then a product in $\gen_2$ and $\gen_2^\ast$, etc. For $i=1,\dotsc,l$, one can simplify the product in $\gen_i$ and $\gen_i^\ast$ by applying the relation $\gen_i^\ast \gen_i=1$ as often as possible. For $i=l+1,\dotsc,n$, one can simplify using that $\gen_i^\ast\gen_i=\gen_i\gen_i^\ast=1$. It is thus seen that $\IUAl$ is spanned by the elements of the form
\[
\gen_1^{e_1}\gen_1^{\ast f_1}\dotsm \gen_l^{e_l}\gen_l^{\ast f_l}\uni_{l+1}^{g_{l+1}}\dotsm \uni_{n}^{g_n}
\]
for $e_1,\dotsc,e_l,f_1,\dotsc,f_l\geq 0$ and $g_{l+1},\dotsc,g_n\in\ZZ$. We then have the following result. Its proof uses \cref{eq:adjoint_of_pure_isometry,eq:adjoint_of_unitary}.

\begin{proposition}\label{res:linear_combination_sent_to_zero}
        Let $x\in\IUAl$ be written as finite linear combination
        \[
        x=\sum_{\genfrac{}{}{0pt}{1}{e_1,f_1,\dotsc,e_l,f_l\geq 0}{ g_{l+1},\dotsc,g_n\in\ZZ}}\lambda_{e_1,f_1,\dotsc,e_l,f_l,g_{l+1},\dotsc,g_n}\gen_1^{e_1}\gen_1^{\ast f_1}\gen_1^{e_1}\gen_1^{\ast f_1}\dotsm \gen_l^{e_l}\gen_l^{\ast f_l}\uni_{l+1}^{g_{l+1}}\dotsm \uni_{n}^{g_n}
        \]
        for some complex constants $\lambda_{e_1,f_1,\dotsc,e_l,f_l,g_{l+1},\dotsc,g_n}$.
        Suppose that $\strep(x)=0$. Then  $\lambda_{e_1,f_1,\dotsc,e_l,f_l,g_{l+1},\dotsc,g_n}=0$ for all $e_i,f_i\geq 0$  and all $g_j\in\ZZ$.
\end{proposition}

\begin{proof}
        Suppose that $l=0$. Since, in particular, $\strep(x)\e_{0,\dotsc,0}\pki=0$, and since $\strep(\uni_1^{g_1}\dotsm\uni_{n}^{g_n})\e_{0,\dotsc,0}\pki=\e_{g_1,\dotsc,g_n}$ for all $g_1,\dotsc,g_n\in\ZZ$, the conclusion in the statement is immediate.

        Suppose that $l$ is such that $1\leq l\leq n$. It is then sufficient to prove the following:
         Let $p\geq 0$, and suppose that $f_1,\dotsc,f_n\geq 0$ are such that $\sum_{i=1}^{l}f_i=p$. Then  $\lambda_{e_1,f_1,\dotsc,e_l,f_l,g_{l+1},\dotsc,g_n}=0$ for all $e_1,\dotsc,e_l\geq 0$  and all $g_{l+1},\dotsc,g_{n}\in\ZZ$. We prove this statement by induction on $p$. For $p=0$, the argument is as follows.

        Write $x$ as
        \[
        x=\sum_{p\geq 0}\!\! \sum_{\genfrac{}{}{0pt}{1}{f_1,\dotsc,f_l\geq 0}{f_1+\dotsb+f_l=p}}\!
        \sum_{\genfrac{}{}{0pt}{1}{e_1,\dotsc,e_l\geq 0}{ g_{l+1},\dotsc,g_n\in\ZZ}}\!\!
        \lambda_{e_1,f_1,\dotsc,e_l,f_l,g_{l+1},\dotsc,g_n} \gen_1^{e_1}\gen_1^{\ast f_1}\dotsm \gen_l^{e_l}\gen_l^{\ast f_l}\uni_{l+1}^{g_{l+1}}\dotsm \uni_{n}^{g_n}.
        \]

       We now let $\strep(x$) act on $\e_{0,\dotsc,0}\pki$. This is annihilated by  $\strep(\gen_1^\ast),\dotsc,\strep( \gen_l^\ast)$, so that
        \[
        \sum_{\genfrac{}{}{0pt}{1}{e_1,\dotsc,e_j\geq 0}{ g_{l+1},\dotsc,g_n\in\ZZ}}
        \lambda_{e_1,0,\dotsc,e_l,0,g_{l+1},\dotsc,g_n}\strep(\gen_1^{e_1}\dotsm \gen_l^{e_l}\uni_{l+1}^{g_{l+1}}\dotsm \uni_{n}^{g_n})\e_{0,\dotsc,0}\pki=0.
        \]
        Since
        \begin{equation}\label{eq:spreading_of_exponents_as_indices}
        \strep(\gen_1^{e_1}\dotsm \gen_l^{e_l}\uni_{l+1}^{g_{l+1}}\dotsm \uni_{n}^{g_n})\e_{0,\dotsc,0}\pki=\e_{e_1,\dotsc,e_l,g_{l+1},\dotsc,g_n},
        \end{equation}
        we see that $\lambda_{e_1,0,\dotsc,e_l,0,g_{l+1},\dotsc,g_n}=0$ for all $e_1,\dotsc,e_l\geq 0$ and all $g_l,\dotsc,g_n\in\ZZ$.
        This proves the statement for $p=0$.

        Assuming the statement for $p_0\geq 0$, we have
        \[
        x=\sum_{p\geq p_0+1}\!\!\! \sum_{\genfrac{}{}{0pt}{1}{f_1,\dotsc,f_l\geq 0}{f_1+\dotsb+f_l=p}}\!
        \sum_{\genfrac{}{}{0pt}{1}{e_1,\dotsc,e_l\geq 0}{ g_{l+1},\dotsc,g_n\in\ZZ}}\!\!\!\!\!\!\!\!\!\lambda_{e_1,f_1,\dotsc,e_l,f_l,g_{l+1},\dotsc,g_n} \gen_1^{e_1}\gen_1^{\ast f_1}\dotsm \gen_l^{e_l}\gen_l^{\ast f_l}\uni_{l+1}^{g_{l+1}}\dotsm \uni_{n}^{g_n}.
        \]
        Fix $\tilde{f}_1,\dotsc,\tilde{f}_l\geq 0$ such that $\sum_{i=1}^{l}\tilde{f}_i=p_0+1$, and consider the action of $\strep(x)$ on $\e_{\tilde{f}_1,\dotsc,\tilde{f}_l,0,\dotsc,0}$.
        If $f_1,\dotsc,f_l\geq 0$ are such that $\sum_{i=1}^{l} f_i> p_0+1$, then there exists $i_0$ with $1\leq i_0\leq l$ such that $f_{i_0}> \tilde{f}_{i_0}$. Then $\strep(\gen_{i_0}^{\ast f_{i_0}})$ annihilates $\e_{\tilde{f}_1,\dotsm\tilde{f}_l,0,\dotsc, 0}$.
        If $f_1,\dotsc,f_l\geq 0$ are such that$\sum_{i=1}^{l} f_i = p_0+1$, but $(f_1,\dotsc,f_l)\not=(\tilde{f}_1,\dotsc,\tilde{f}_l)$, then again there exists $i_0$ with $1\leq i_0\leq l$ such that $f_{i_0}>\tilde{f}_{i_0}$, and again $\strep(\gen_{i_0}^{\ast f_{i_0}})$ annihilates $\e_{\tilde{f}_1,\dotsm,\tilde{f}_l,0,\dotsc, 0}$.
        Hence
        \[
        \sum_{\genfrac{}{}{0pt}{1}{e_1,\dotsc,e_l\geq 0}{ g_{l+1},\dotsc,g_n\in\ZZ}}\!\!\!\!\!
        \lambda_{e_1,\tilde{f}_1,\dotsc,e_l,\tilde{f}_l,g_{l+1},\dotsc,g_n}    \strep(\gen_1^{e_1}\gen_1^{\ast \tilde{f}_1}\dotsm \gen_l^{e_l}\gen_l^{\ast \tilde{f}_l}\uni_{l+1}^{g_{l+1}}\dotsm \uni_{n}^{g_n}) \e_{\tilde{f}_1\dotsm\tilde{f}_l,0,\dotsc,0}=0.
        \]
        The words $\gen_1^{e_1}\gen_1^{\ast \tilde{f}_1}\dotsm \gen_l^{e_l}\gen_l^{\ast \tilde{f}_l}\uni_{l+1}^{g_{l+1}}\dotsm \uni_{n}^{g_n}$ occurring in this summation can be rewritten to end with $\gen_1^{\ast \tilde{f}_1}\dotsm \gen_l^{\ast \tilde{f}_l}$ at the cost of a unimodular constant. Since, furthermore, $\pi(\gen_1^{\ast \tilde{f}_1}\dotsm \gen_l^{\ast \tilde{f}_l}) \e_{\tilde{f}_1\dotsm \tilde{f}_l,0,\dotsc,0}$ equals $\e_{0,\dotsc,0}\pki$ up to a unimodular constants, \cref{eq:spreading_of_exponents_as_indices} then shows that
       $\lambda_{e_1,\tilde{f_1},\dotsc,e_l,\tilde{f_l},g_{l+1},\dotsc,g_n}=0$ for all $e_1,\dotsc,e_l\geq 0$ and all $g_{l+1},\dotsc,g_n\in\ZZ$.
        Since $\tilde{f}_1,\dotsc,\tilde{f}_l\geq 0$ were arbitrary subject to the condition $\sum_{i=1}^{l}\tilde{f}_i=p_0+1$, this completes the induction step.
\end{proof}

The following consequences of \cref{res:linear_combination_sent_to_zero} are immediate.
\begin{corollary}\label{res:basis_and_injectivity}\quad
        \begin{enumerate}
                \item The set
                \[
                \desset{
                \gen_1^{e_1}\gen_1^{\ast f_1}\dotsm \gen_l^{e_l}\gen_l^{\ast f_l}\uni_{l+1}^{g_{l+1}}\dotsm \uni_{n}^{g_n}: e_1,f_1,\dotsc,e_l,f_l\geq 0,\,  g_{l+1},\dotsc,g_n\in\ZZ}
                \]
                is a basis of $\IUAl$.
                \item The standard representation $\strep$ of $\IUAl$ on $\Hl$ is injective.
                \item $\IUAl$ can canonically be identified with a dense unital involutive subalgebra of $\UAl$.
        \end{enumerate}
\end{corollary}

\subsection{An injective representation of a $\mathrm{C}^\ast$-subalgebra of $\UAl$}\label{subsec:injective_representation_of_subalgebra} 
	
We let $\IPAl$ denote the unital involutive subalgebra of $\IUAl\subseteq\UAl$ that is generated by the set of (range) projections
\[
\desset{\gen_i^k \gen_i^{\ast k} : i=1,\dotsc,l,\,k \geq 0}.
\]
When $l=0$, $\IPAl$ consists of the multiples of the identity. We let $\PAl$ denote the closure of $\IPAl$ in $\UAl$. It follows from \cref{res:algebraic_results} that $\IPAl$ and $\PAl$ are commutative.
The second step in our programme consists of finding a particular unital injective representation of $\PAl$, which we shall use later (via a projection from $\UAl$ onto $\PAl$) to define a faithful unital completely positive map from $\UAl$ to the bounded operators on a Hilbert space. This injective unital representation of $\PAl$ is the following.

Take the standard representation $\strep$ of $\UAl$ on $\Hl$, and let $\ssp$ be the closed linear span $\ssp$ of $\set{\e_{k_1,\dotsc,k_l, 0, \dotsc, 0}: k_1,\dotsc,k_l\geq 0}$ in $\Hl$; when $l=0$ this is simply the span of $ \e_{0,\dotsc,0}$. In order to shorten the notation, we set
\[
\ee_{k_1,\dotsc,k_l}\coloneqq \e_{k_1,\dotsc,k_l, 0, \dotsc, 0}
\] for $k_1,\dotsc,k_l\geq 0$. Then \cref{eq:general_isometries,eq:adjoint_of_pure_isometry} show that, for $i=1,\dotsc,l$ and $k_1,\dotsc,k_l\geq 0$,
\begin{equation}\label{eq:restricted_action_of_isometries}
	\strep(\gen_i\pa)\ee_{k_1,\dotsc,k_{i-1}, k_{i},k_{i+1},\dotsc,k_l}\coloneqq \z_{i,1}^{k_1}\dotsm\z_{i,i-1}^{k_{i-1}} \ee_{k_1,\dotsc,k_{i-1},k_i+1,k_{i+1},\dotsc, k_l},
\end{equation}
and
\begin{equation}\label{eq:restricted_action_of_adjoints_of_isometries}
	\strep(\gen_i^\ast) \ee_{k_1,\dotsc,k_l} =
	\begin{cases}
		\zz_{i,1}^{k_1}\dotsm\zz_{i,i-1}^{k_{i-1}} \ee_{k_1, \dotsc, k_{i-1}, k_i-1, k_{i+1} ,\dotsc, k_l} & \textup{if } k_i\geq 1,\\
		0& \textup{if } k_i=0.
	\end{cases}
\end{equation}

Hence $\ssp$ affords a representation of the \Cstar-subalgebra of $\UAl$ that is generated by $\gen_1,\dotsc,\gen_l$ and their adjoints. In particular, we can define a unital representation $\resstrep:\PAl\to \bounded(\ssp)$ of $\PAl$ on $\ssp$ by setting
\[
\resstrep(x)h:=\strep(x)h
\]
for $x\in\PAl$ and $h\in \ssp$. We shall prove that $\resstrep$ is injective by showing that its restriction to the dense subalgebra $\IPAl$ of $\PAl$ is isometric.

As a preparation for this, we note that the restriction of $\rho$ to $\IPAl$ is injective. To see this, we let $\mathcal B^0$ denote the involutive subalgebra of $\IUAl$ that is generated by $\gen_1,\dotsc,\gen_l$ and their adjoints. It follows from a double application of the first part of \cref{res:basis_and_injectivity} that $\mathcal B$ is the universal unital involutive algebra that is generated by elements $s_1,\dotsc,s_l$ satisfying
\cref{eq:relations_in_universal_algebra,eq:automatic_relations_in_universal_algebra,eq:isometries_in_universal_algebra}, but then for indices between 1 and $l$. Since the restricted standard representation of $\mathcal B^0$ to $\ssp$ is obviously equivalent to the standard representation of $\mathcal B^0$, the second part of \cref{res:basis_and_injectivity} shows that it is injective on the entire $\mathcal B^0$, and then in particular on $\IPAl$.

\begin{lemma}\label{res:element_is_linear_combination_of_projections}
        Let $x$ be an element of $\IPAl$. Then there exist pairwise orthogonal projections $p_1,\dotsc,p_N$ in $\IPAl$ and $\lambda_1,\dotsc,\lambda_N\in\CC$ such that $x=\sum_{i=1}^N \lambda_i p_i$.
\end{lemma}

\begin{proof}
        When $l=0$ this is clear, so we suppose that $1\leq l\leq n$.
        Take $x$ in $\IPAl$. \cref{res:algebraic_results} shows that it can be written as
        \begin{equation}\label{eq:decomposition}
        x=\sum _{e_1,\dotsc,e_l=0}^N \lambda_{e_1,\dotsc,e_l} \iso_1^{e_1}s_1^{\ast e_1}\dotsm \iso_l^{e_l}s_l^{\ast e_l}
        \end{equation}
        with $\lambda_{e_1,\dotsc,e_l}\in\CC$. We may suppose that $N\geq 1$. For $i=1,\dotsc,l$ and $k=0,\dotsc,N-1$, we define $q_i(k)\in\IPAl$ by setting
        \[
        q_i(k)\coloneqq \iso_i^{k}\iso_i^{\ast k}-\iso_i^{k+1}\iso_i^{\ast(k+1)};
        \]
        for $i=1,\dotsc,l$, we define $q_i(N)\in\IPAl$ by setting
        \[
        q_i(k)\coloneqq \iso_i^{N}\iso_i^{\ast N}.
        \]
        For a fixed $i$, it follows from \cref{res:algebraic_results} that the  $q_i(0),\dotsc,q_i(N)$ are pairwise orthogonal projections.

        Take a product $\iso_1^{e_1}s_1^{\ast e_1}\dotsm \iso_l^{e_l}s_l^{\ast e_l}$ occurring in the summation in \cref{eq:decomposition}. Each factor $\iso_i^{e_i} \iso_i^{\ast e_i}$ in it is a linear combination (in fact a telescoping sum) of $q_i(0),\dotsc,q_i(N)$. Hence each of these products, and then also $x$, is a linear combination of products of the form $q_1(\alpha_1)\dotsb q_l(\alpha_l)$ where $0\leq\alpha_1,\dotsc,\alpha_l\leq N$. Since $\IPAl$ is commutative, these  $(N+1)^l$ products are projections, and the pairwise orthogonality of the $q_i(k)$ for a fixed $i$ shows that they are pairwise orthogonal.
     \end{proof}

\begin{lemma}\label{res:norm_of_linear_combination_of_projections}\quad
        \begin{enumerate}
                \item Let $x$ be an element of a \Cstar-algebra $\mathcal A$ such that $x=\sum_{j=1}^{N} \lambda_j p_j$ for some $\lambda_1,\dotsc,\lambda_N\in\CC$ and pairwise orthogonal projections $p_1,\dotsc,p_N$ in $\mathcal A$.
                Then $\norm{x}\leq \max_{1\leq j\leq N} |\lambda_j|$.
                \item Let $H$ be a Hilbert space, let $P_1,\dotsc,P_N$ be mutually orthogonal non-zero projections on $H$, and let $\lambda_1,\dotsc,\lambda_N\in\CC$. Then $\lrnorm{\sum_{j=1}^{N} \lambda_jP_j}=\max_{1\leq j\leq N}|\lambda_j|$.
        \end{enumerate}
\end{lemma}

\begin{proof}
	The first part follows easily from the Pythagorean theorem once $\mathcal A$ has been realised as \Cstar-algebra of operators on a Hilbert space. Alternatively, one can observe that, since $a$ is normal, the norm of $a$ equals its spectral radius. Hence
        \begin{align*}
        \norm{x} &= \lim_{n\to \infty} \lrnorm{\left(\sum_{j=1}^{N} \lambda_jp_j\right)^n}^{\frac{1}{n}}\\
        &= \lim_{n\to \infty} \lrnorm{\sum_{j=1}^{N} \lambda_j^np_j}^{\frac{1}{n}}.
        \end{align*}
        Since $\lrnorm{\sum_{j=1}^{N} \lambda_j^np_j}^{\frac{1}{n}} \leq \left(N(\max_{1\leq j\leq N}|\lambda_j|)^n\right)^{\frac{1}{n}}$, the first part follows.

For the second part, we have $\lrnorm{\sum_{j=1}^{N} \lambda_jP_j}\leq\max_{1\leq j\leq N} |\lambda_j|$ from the first part.
        Consideration of the action of $ \sum_{j=1}^{N} \lambda_j P_j$ on the (non-zero) subspace $P_jH$ of $H$ shows that $\lrnorm{\sum_{j=1}^{N} \lambda_jP_j}\geq |\lambda_j|$ for all $j$ such that $1\leq j\leq N$.
\end{proof}

\begin{lemma}\label{res:rho_isometric_on_subalgebra}
	Let $x$ be an element of $\IPAl$. Then $\norm{\resstrep(x)}=\norm{x}$.
\end{lemma}

\begin{proof}
        We may suppose that $x\neq 0$. According to \cref{res:element_is_linear_combination_of_projections},
        there exist pairwise orthogonal projections $p_1,\dotsc,p_N$ in $\IPAl$ and $\lambda_1,\dotsc,\lambda_N\in\CC$ such that $x=\sum_{i=1}^N \lambda_i p_i$. We may suppose that $p_i\neq 0$ for all $i$. Since $\resstrep$ is injective, we also have $\resstrep(p_i)\neq 0$ for all $i$.
        Then $\norm{\resstrep(x)}=\max_{1\leq j\leq N}|\lambda_j|$ by the second part of \cref{res:norm_of_linear_combination_of_projections}, whereas its first part shows that $\norm{x} \leq \max_{1\leq j\leq N}|\lambda_j|$. Since certainly $\norm{\resstrep(x)}\leq\norm{x}$, we have
        \[
        \max_{1\leq j\leq N}|\lambda_j| = \norm{\resstrep(x)}\leq\norm{x}\leq \max_{1\leq j\leq N}|\lambda_j|.
        \]
\end{proof}

We have now reached the goal of this subsection.

\begin{corollary}\label{res:rho_injective_on_subalgebra}
        The unital representation $\resstrep: \PAl\to \bounded(\ssp)$ of $\PAl$ on $\ssp$ that is determined by \cref{eq:restricted_action_of_isometries,eq:restricted_action_of_adjoints_of_isometries} is injective.
\end{corollary}

\subsection{A faithful unital completely positive map on $\UAl$}\label{subsec:faithful_completely_positive_map}

Using the injective unital representation $\resstrep$ of $\PAl$ on $\ssp$ from \cref{subsec:injective_representation_of_subalgebra}, it is not difficult anymore to find a faithful unital completely positive map from $\UAl$ to the bounded operators on a Hilbert space. The only extra ingredient that is needed is a unital norm one projection from $\UAl$ onto $\PAl$, and this is easily found.

For $t=(t_1,\dotsc,t_n)\in\TT^n$, the elements $t_1\gen_1,\dotsc,t_n\gen_n$ of $\UAl$ still satisfy the relations \cref{eq:relations_in_universal_algebra,eq:automatic_relations_in_universal_algebra,eq:isometries_in_universal_algebra,eq:unitaries_in_universal_algebra}. Hence there exists an automorphism $\alpha_t:\UAl\to\UAl$ such that $\alpha_t(\gen_i)=t_i\gen_i$ for $i=1,\dotsc,n$. The resulting action of $\TT^n$ on $\UAl$ is easily seen to be strongly continuous, and this enables us to define a linear map $\theta:\UAl\to\UAl$ by setting
\begin{equation}\label{eq:theta}
\theta(x)\coloneqq \int_{\TT^n}\!\!\!\!\!\alpha_t(x)\di{t},
\end{equation}
where the total measure of $\TT^n$ equals 1. Then $\theta$ is unital and bounded, and $\norm{\theta}=1$. A moment's thought shows that, for $e_1,f_1,\dotsc,e_l,f_l\geq 0$ and $g_{l+1},\dotsc,g_{n}\in\ZZ$,
\begin{align}
\theta(\gen_1^{e_1}\gen_1^{\ast f_1}\dotsm \gen_l^{e_l}\gen_l^{\ast f_l} \uni_{l+1}^{g_{l+1}}\dotsc\uni_{n}^{g_n})&=0\label{eq:theta_on_words_one}\\
\intertext{unless $e_1=f_1,\dotsc,e_l=f_l$ and $g_{l+1}=\dotsc=g_{n}=0$, and that}
\theta(\gen_1^{e_1}\gen_1^{\ast e_1}\dotsm \gen_l^{e_l}\gen_l^{\ast e_l})&=\gen_1^{e_1}\gen_1^{\ast e_1}\dotsm \gen_l^{e_l}\gen_l^{\ast e_l}\label{eq:theta_on_words_two}.
\end{align}
Since the products $\gen_1^{e_1}\gen_1^{\ast f_1}\dotsm \gen_l^{e_l}\gen_l^{\ast f_l}$ span the dense subalgebra $\IUAl$ of $\UAl$, and the products $\gen_1^{e_1}\gen_1^{\ast e_1}\dotsm \gen_l^{e_l}\gen_l^{\ast e_l}$ span the dense subalgebra $\IPAl$ of $\PAl$, $\theta$ is a norm one projection $\theta: \UAl\to\PAl$.

\begin{proposition}\label{res:theta_properties}
	The unital map $\theta: \UAl\to \PAl$ is completely positive and faithful.
\end{proposition}

\begin{proof}
	Since $\theta$ is a norm one projection from a \Cstar-algebra onto a \Cstar-subalgebra, Tomiyama's theorem (see \cite[II.6.10.2]{blackadar_OPERATOR_ALGEBRAS:2006}, for example) shows that $\theta$ is completely positive.
	We need to prove only that $\theta$ is faithful.
	Take a non-zero $x$ in $\UAl$, and next a state $\tau$ on $\UAl$ such that $\tau(x^\ast x)\not=0$.
	Then
	\[
	\tau(\theta(x^\ast x))=\int_{\TT^n}\!\!\!\!\!\tau(\alpha_t(x^\ast x)) \di{t}.
	\]
	Since the continuous integrand is non-negative, and strictly positive for $t=(1,\dotsc,1)$, the integral is strictly positive. Hence $\theta(x^\ast x)\neq 0$.
\end{proof}

\begin{remark}If $l=0$, then $\theta$ is tracial, as remarked on \cite[p.4]{rieffel:1990}. If  $l\geq 1$, then $\theta$ is not tracial because $\theta(\iso_1\pa \iso_1^{\ast 2} \cdot \iso_1^2 \iso_1^{\ast}) =\theta(
	\iso_1\pa \iso_1^{\ast})=\iso_1\pa \iso_1^{\ast}$, whereas $\theta(\iso_1^2 \iso_1^{\ast}\cdot \iso_1\pa \iso_1^{\ast 2})=\theta(\iso_1^2 \iso_1^{\ast 2})=\iso_1^2 \iso_1^{\ast 2}$.
\end{remark}

We now use the representation $\resstrep$ of $\PAl$ on the Hilbert space $L_l$ with orthonormal basis $\ee_{k_1,\dotsc,k_l}$ for $k_1,\dotsc,k_l\geq 0$ from \cref{subsec:injective_representation_of_subalgebra} to define
\[
\phi:\UAl\to \bounded(\ssp)
\]
by setting
\[
\phi\coloneqq \resstrep\circ\theta.
\]

This $\phi$ has the desired properties for an application of Stinespring's theorem in \cref{subsec:standard_representation_is_injective}.

\begin{theorem}\label{res:phi_properties}
	The map $\phi:\UAl\to \bounded(\ssp)$ is unital, faithful, and completely positive. Its image $\phi(\UAl)$ is the unital commutative \Cstar-subalgebra of $\bounded(\ssp)$ that is generated by the elements $\rho(\iso_i^k\iso_i^{\ast k})$ for $i=1,\cdots,l$ and $k\geq 0$, i.e., by the range projections of $\rho(s_i^k)$ for $i=1,\cdots,l$ and $k\geq 0$.
\end{theorem}

\begin{proof}
	Since $\theta:\UAl\to\PAl$ is unital and completely positive, and since $\resstrep:
	\PAl\to \bounded(\ssp)$ is a unital representation, the composition $\phi$ is also unital and completely positive.
	It is faithful since $\resstrep$ is injective on $\PAl$ and $\theta$ is faithful on $\UAl$.

	The final statement is clear from the construction in \cref{subsec:injective_representation_of_subalgebra}.
\end{proof}

\begin{remark}
	It is an additional consequence of Tomiyama's theorem that $\phi(x_1yx_2)=\resstrep(x_1)\phi(y)\resstrep(x_2)$ for $x_1,x_2\in\PAl$ and $y\in\UAl$, but we shall not need this.
\end{remark}

\subsection{The standard representation of ${\UAl}$ is injective}\label{subsec:standard_representation_is_injective}

We now apply Stinespring's Theorem (see \cite[II.6.9.7]{blackadar_OPERATOR_ALGEBRAS:2006}, for example) to the faithful unital completely positive map $\phi:\UAl\to \bounded(\ssp)$. The pertinent construction yields a representation $\pi_\phi$ of $\UAl$ on a Hilbert space $H_\phi$ containing $\ssp$, such that
\begin{equation}\label{eq:stinespring_relation}
\phi(x)=P_{\ssp} \pi_\phi(x)P_{\ssp}\
\end{equation}
for  $x\in\UAl$; here $P_{\ssp}:H_\phi\to \ssp$ is the associated projection. By \cref{eq:stinespring_relation}, the fact that $\phi$ is faithful implies that $\pi_\phi$ is injective. The remainder of this subsection is concerned with showing that $\pi_\phi$ is essentially an inflation of the standard representation $\strep$ of $\UAl$ on $\Hl$, which is, therefore, also injective.

In order to analyse $\pi_\phi$, we recall how it is obtained.

One starts by introducing a semi-inner product on $\UAl\otimes\ssp$ such that
\[
\ipstsp{x_1\otimes h_1, x_2\otimes h_2}=\ipssp{\phi(x_2^\ast x_1)h_1, h_2}
\]
for $x_1,x_2\in\UAl$ and $h_1,h_2\in\ssp$. The Hilbert space $H_\phi$ is then the completion of $\UAl\otimes\ssp$ modulo the isotropic vectors. Using square brackets to indicate elements of the quotient, we therefore have
\begin{align*}
\ipstsp{ [x_1\otimes h_1], [x_2\otimes h_2]}&=\ipssp{ \phi(x_2^\ast x_1)h_1, h_2}
\intertext{for $x_1,x_2\in\UAl$ and $h_1,h_2\in\ssp$. The representation}
\pi_\phi:\UAl&\to \bounded(H_\phi)
\intertext{is such that}
\pi_\phi(y)[x\otimes h]&= [yx\otimes h]
\end{align*}
for $x,y\in\UAl$ and $h\in\ssp$.

When $x_1,x_2\in\UAl$ and $h_1,h_2\in\ssp$ are such that $[x_1\otimes h_1]=[x_2\otimes h_2]$, then $[yx_1\otimes h_1]=[yx_2\otimes h_2]$ for all $y\in\UAl$ because $[yx_1\otimes h_1]=\pi_\phi(y)[x_1\otimes h_1]=\pi_\phi(y)[x_2\otimes h_2]=[yx_2\otimes h_2]$. We shall use this elementary observation a few times in the remainder of this subsection.

Using that $\UAl$ is the completion of $\IUAl$, it is easily seen that the elements
\[
[\iso_1^{e_1}\iso_1^{\ast f_1}\dotsm \iso_l^{e_l}\iso_l^{\ast f_l} \uni_{l+1}^{g_{l+1}}\dotsm \uni_{n}^{g_n}\otimes \ee_{k_1,\dotsc,k_l}]
\]
for $e_1,f_1,\dotsc,e_l,f_l\geq 0$, $g_{l+1},\dotsc, g_n\in\ZZ$, and $k_1,\dotsc,k_l\geq 0$ span a dense subspace of $H_\phi$.

For $k_1,\dotsc,k_l\geq 0$, let $L_{k_1,\dotsc,k_l}$ be the closed linear span in $H_\phi$ of
the elements
\[
	[\iso_1^{e_1}\iso_1^{\ast f_1}\dotsm \iso_l^{e_l}\iso_l^{\ast f_l} \uni_{l+1}^{g_{l+1}}\dotsm \uni_{n}^{g_n}\otimes \ee_{k_1,\dotsc, k_l}]
\]
for $e_1,f_1,\dotsc,e_l,f_l\geq 0, \, g_{l+1},\dotsc,g_n\in\ZZ$. Since
\[
L_{k_1,\dotsc,k_l}=\overline{\pi_\phi(\UAl)(1\otimes\ee_{k_1,\dotsc,k_l})},
\]
it is clear that $L_{k_1,\dotsc, k_n}$ is invariant under the action of $\UAl$.

\begin{lemma}\label{res:H_is_direct_sum}
        The Hilbert space $H_\phi$ is the Hilbert direct sum of its subspaces $L_{k_1,\dotsc, k_l}$ for $k_1,\dotsc,k_l\geq 0$. Each of these reduces $\pi_\phi$.
\end{lemma}

\begin{proof}
	For the first statement, we need to show only that the subspaces $L_ {k_1,\dotsc,k_l}$ and $L_ {k_1^\prime,\dotsc,k_n^\prime}$ are orthogonal when $(k_1,\dotsc,k_l)\neq(k_1^\prime,\dotsc,k_l^\prime)$. Take such different $l$-tuples. It is sufficient to show that
	\[
	\ipstsp{[x\otimes \ee_{k_1,\dotsc, k_l}], [y\otimes \ee_{k_1^\prime,\dotsc, k_l^\prime}]}=0\]
	for $x,y\in\UAl$. Now
        \[
        \ipstsp{[x\otimes \ee_{k_1,\dotsc, k_l}], [y\otimes \ee_{k_1^\prime,\dotsc, k_l^\prime}]}=\ipssp{
        	\phi(y^\ast x)\ee_{k_1,\dotsc, k_l}, \ee_{k_1^\prime,\dotsc, k_l^\prime}}.
        \]
        Since $\phi(y^\ast x)$ is in the \Cstar-subalgebra of $\bounded(\ssp)$ that is generated by (the restrictions to $H_l$ of) $\strep(\iso_i^k)\strep(\iso_i^{\ast k})$ for $i=1,\dotsc,l$ and $k\geq 0$, it is clear from \cref{eq:restricted_action_of_isometries,eq:restricted_action_of_adjoints_of_isometries} that  $\phi(y^\ast x)\ee_{k_1,\dotsc, k_l}$ is a multiple of $\ee_{k_1,\dotsc, k_l}$. Hence the two subspaces are orthogonal.

        The final statement is now clear.
\end{proof}

The next step is to show that the representation of $\UAl$ on each of the subspaces $L_ {k_1,\dotsc,k_n}$ is unitarily equivalent to the standard representation. We start with the necessary preparatory results for this.

\begin{lemma}\label{res:if_nonzero_then_norm_one} Take $k_1,\dotsc,k_l\geq 0$. Let $e_1,f_1,\dotsc,e_l,f_l\geq 0$ and $g_{l+1},\dotsc,g_n\in\ZZ$. Then
	\[
	\lrnorm{[\iso_1^{e_1}\iso_1^{\ast f_1}\dotsm \iso_l^{e_l}\iso_l^{\ast f_l} \uni_{l+1}^{g_{l+1}}\dotsm \uni_{n}^{g_{n}}\otimes \ee_{k_1,\dotsc,k_l}]}_\phi=1
	\]
when $0\leq f_1\leq k_1,\dotsc, 0\leq f_l\leq k_l$. Otherwise, this norm is zero.
\end{lemma}

\begin{proof}
        Using \cref{res:algebraic_results} in the final step, we see that
        \begin{align*}
        &(\iso_1^{e_1}\iso_1^{\ast f_1}\dotsm \iso_l^{e_l}\iso_l^{\ast f_l} \uni_{l+1}^{g_{l+1}}\dotsm \uni_{n}^{g_n})^\ast\cdot (\iso_1^{e_1}\iso_1^{\ast f_1}\dotsm \iso_l^{e_l}\iso_l^{\ast f_l} \uni_{l+1}^{g_{l+1}}\dotsm \uni_{n}^{g_n})\\
        &=\uni_{n}^{\ast g_n}\dotsm \uni_{l+1}^{\ast g_{l+1}} \iso_l^{f_l}\iso_l^{\ast e_l}\dotsm \bm{\iso_1^{f_1}\iso_1^{\ast e_1} \cdot \iso_1^{e_1}\iso_1^{\ast f_1}}\dotsm \iso_l^{e_l}\iso_l^{\ast f_l} \uni_{l+1}^{g_{l+1}}\dotsm \uni_{n}^{g_n}\\
        &=\uni_{n}^{\ast g_n}\dotsm \uni_{l+1}^{\ast g_{l+1}} \iso_l^{f_l}\iso_l^{\ast e_l}\dotsm \bm{\iso_1^{f_1}\iso_1^{\ast f_1}}\dotsm \iso_l^{e_l}\iso_l^{\ast f_l} \uni_{l+1}^{g_{l+1}}\dotsm \uni_{n}^{g_n}\\
        &=\bm{\iso_1^{f_1}\iso_1^{\ast f_1}} \cdot \uni_{n}^{\ast g_n}\dotsm \uni_{l+1}^{\ast g_{l+1}} \iso_l^{f_l}\iso_l^{\ast e_l}\dotsm \iso_2^{f_2}\iso_2^{\ast e_2} \cdot \iso_2^{e_2}\iso_2^{\ast f_2}\dotsm\iso_l^{e_l}\iso_l^{\ast f_l} \uni_{l+1}^{g_{l+1}}\dotsm \uni_{n}^{g_n}.
        \end{align*}
        We apply this procedure $(l-1)$ more times, after which we use that $\uni_{n}^{\ast g_n}\dotsm \uni_{l+1}^{\ast g_{l+1}} \cdot \uni_{l+1}^{g_{l+1}}\dotsm \uni_{n}^{g_n} \cdots \uni_{n}^{g_n}=1$, so that we end up with $\iso_1^{f_1}\iso_1^{\ast f_1}\dotsm \iso_l^{f_l}\iso_l^{\ast f_l}$.
        Therefore,
        \begin{align*}
        \lVert[&\iso_1^{e_1}\iso_1^{\ast f_1}\dotsm \iso_l^{e_l}\iso_l^{\ast f_l} \uni_{l+1}^{g_{l+1}}\dotsm \uni_{n}^{g_n}\otimes \ee_{k_1,\dotsc, k_l}]\rVert_\phi^2\\
        &=\ipssp{\phi(\iso_1^{f_1}\iso_1^{\ast f_1}\dotsm \iso_l^{f_l}\iso_l^{\ast f_l}) \ee_{k_1,\dotsc, k_l}, \ee_{k_1,\dotsc, k_l}}\\
        &=\ipssp{\resstrep(\theta(\iso_1^{f_1}\iso_1^{\ast f_1}\dotsm \iso_l^{f_l}\iso_l^{\ast f_l})) \ee_{k_1,\dotsc, k_l}, \ee_{k_1,\dotsc, k_l}}\\
        &=\ipssp{\resstrep(\iso_1^{f_1}\iso_1^{\ast f_1}\dotsm \iso_l^{f_l}\iso_l^{\ast f_l}) \ee_{k_1,\dotsc, k_l}, \ee_{k_1,\dotsc, k_l}}\\
        &= \ipssp{\resstrep(\iso_1^{f_1}\iso_1^{\ast f_1})\dotsm \resstrep(\iso_l^{f_l}\iso_l^{\ast f_l}) \ee_{k_1,\dotsc, k_l}, \ee_{k_1,\dotsc,k_l}}.
        \end{align*}
        Since, for $i=1,\dotsc,l$, $\resstrep(\iso_i^{f_i}\iso_i^{\ast f_i})$ is the restriction to $H_l$ of the range projection of $\strep(s_i^{f_i})$, the statement is then obvious from \cref{eq:restricted_action_of_isometries}.
\end{proof}

\begin{proposition}\label{res:action_equivalent_to_standard_representation}\quad
        \begin{enumerate}
                \item The elements
                \[ [\iso_1^{e_1}\dotsm \iso_l^{e_l}\uni_{l+1}^{g_{l+1}}\dotsm \uni_{n}^{g_n}\otimes \ee_{0,\dotsc, 0}]
                \]
                for $e_1,\dotsc,e_l\geq 0$ and $g_{l+1},\dotsc,g_n\in\ZZ$ form an orthonormal basis of $L_{0,\dotsc, 0}$.
                \item The subrepresentation of $\pi_\phi$ of $\UAl$ on $L_{0,\dotsc, 0}$ is unitarily equivalent to the standard representation $\strep$ of $\UAl$ on $\Hl$.
        \end{enumerate}
\end{proposition}

\begin{proof}
       For the first part, we note that \cref{res:if_nonzero_then_norm_one} implies that the elements of the proposed basis have norm one, and also that they span a dense subspace of $L_{0,\dotsc,0}$ since the remaining elements in the spanning set for $L_{0,\dotsc,0}$ are all zero. Hence we need to show only that the elements of the proposed basis are pairwise orthogonal. This follows easily from the definition of $\theta$.

For the second part, we use \cref{rem:mnemonic} to see that the unitary operator that, for $e_1,\dotsc,e_l\geq 0$ and $g_{l+1},\dotsc,g_{n}\in\ZZ$, sends $[\iso_1^{e_1}\dotsm \iso_l^{e_l}\uni_{l+1}^{g_{l+1}}\dotsm \uni_{n}^{g_n}\otimes \ee_{0,\dotsc, 0}]$ to the element $\e_{e_1,\dotsc,e_l, g_{l+1},\dotsc,g_n}$ of $\Hl$, is a unitary equivalence between the two representations.
\end{proof}

Now that $L_{0,\dotsc,0}$ has been taken care of in \cref{res:action_equivalent_to_standard_representation}, 
 we turn to the general space $L_{k_1,\dotsc, k_l}$.

\begin{lemma}\label{res:existence_of_unimodular_constant} Suppose that $e_1,\dotsc,e_l$ and $k_1,\dotsc,k_l$ are such that $0\leq e_i\leq k_i$ for $i=1,\dotsc,l$. Then there exists a unimodular constant $c$ such that
        \[
        [\iso_1^{e_1}\dotsm \iso_l^{e_l}\iso_1^{\ast k_1}\dotsm \iso_l^{\ast k_l}\otimes \ee_{k_1,\dotsc, k_l}]=c[\iso_1^{\ast (k_1 -e_1)}\dotsm \iso_l^{\ast (k_l-e_l)}\otimes \ee_{k_1,\dotsc, k_l}].
        \]
\end{lemma}

\begin{proof}
        There exists an unimodular constant $\lambda$ such that
        \[
        [\bm{\iso_1^{e_1}}\dotsm \iso_l^{e_l}\iso_1^{\ast k_1}\dotsm \iso_l^{\ast k_l}\otimes \ee_{k_1,\dotsc, k_l}]=\lambda [\iso_2^{e_2}\dotsm \iso_l^{e_l}\bm{\iso_1^{e_1}}\iso_1^{\ast k_1}\dotsm \iso_l^{\ast k_l}\otimes \ee_{k_1,\dotsc, k_l}].
        \]
        Now $\lrnorm{[\iso_1^{\ast k_1}\dotsm \iso_l^{\ast k_l}\otimes \ee_{k_1,\dotsc, k_l}]}_\phi=1$ by \cref{res:if_nonzero_then_norm_one}.
        Since $\pi_\phi(s_1)$ is an isometry on $H_\phi$, this implies that
        \[
        \lrnorm{[\iso_1^{e_1}\iso_1^{\ast k_1}\dotsm \iso_l^{\ast k_l}\otimes \ee_{k_1,\dotsc, k_l}]}_\phi=1
        \]
        Hence, using \cref{res:if_nonzero_then_norm_one} again in the third step,
        \begin{align*}
        \big\lVert [\iso_1^{e_1} & \iso_1^{\ast k_1}\dotsm \iso_l^{\ast k_l}\otimes \ee_{k_1,\dotsc, k_l}\ -\ \iso_1^{\ast (k_1-e_1)}\iso_2^{\ast k_2}\dotsm \iso_l^{\ast k_l}\otimes \ee_{k_1,\dotsc, k_l} ]\big\rVert_\phi^2\\
        &=\lrnorm{[\iso_1^{e_1}\iso_1^{\ast k_1}\dotsm \iso_l^{\ast k_l}\otimes \ee_{k_1,\dotsc, k_l}}_\phi^2+\lrnorm{\iso_1^{\ast (k_1-e_1)}\dotsm \iso_l^{\ast k_l}\otimes \ee_{k_1,\dotsc, k_l} ]}_\phi^2\\
        &\phantom{=}\,-\ipssp{ \phi(\iso_l^{k_l}\dotsm \iso_2^{k_2}\bm{\iso_1^{k_1-e_1}\cdot\iso_1^{e_1}\iso_1^{\ast k_1}} \iso_2^{\ast k_2}\dotsm \iso_l^{\ast k_l}) \ee_{k_1,\dotsc, k_l}, \ee_{k_1,\dotsc, k_l}}\\
        &\phantom{=}\,-\overline{\ipssp{\phi(\iso_l^{k_l}\dotsm \iso_2^{k_2}\bm{\iso_1^{k_1-e_1}\cdot\iso_1^{e_1}\iso_1^{\ast k_1} }\iso_2^{\ast k_2}\dotsm \iso_l^{\ast k_l})\ee_{k_1,\dotsc, k_l}, \ee_{k_1,\dotsc, k_l}}}\\
        &=1+1-\lrnorm{\iso_1^{\ast k_1}\dotsm \iso_l^{\ast k_l}\otimes \ee_{k_1,\dotsc, k_l}}_\phi^2-\overline{\lrnorm{\iso_1^{\ast k_1}\dotsm \iso_l^{\ast k_l}\otimes \ee_{k_1,\dotsc, k_l}}_\phi^2}\\
        &=2-1-1\\
        &=0.
        \end{align*}
        Using this in the first step, and \cref{res:algebraic_results} in the second, we see that
        \begin{align*}
        [\iso_1^{e_1}\dotsm \iso_l^{e_l}\iso_1^{\ast k_1}\dotsm \iso_l^{\ast k_l}\otimes \ee_{k_1,\dotsc, k_l}]&= \lambda[\iso_2^{e_2}\dotsm \iso_l^{e_l}\iso_1^{\ast (k_1-e_1)} \iso_2^{\ast k_2}\dotsm \iso_l^{\ast k_l}\otimes \ee_{k_1,\dotsc, k_l}] \\
        &= \lambda^\prime [\iso_1^{\ast (k_1-e_1)} \iso_2^{e_2}\dotsm \iso_l^{e_l}\iso_2^{\ast k_2}\dotsm \iso_l^{\ast k_l}\otimes \ee_{k_1,\dotsc, k_l}] \nonumber
        \end{align*}
        for some unimodular constant $\lambda^\prime$.
        Now we repeat this procedure for $s_2,\dotsc,s_l$ to arrive at the result.
\end{proof}

\begin{lemma}\label{res:elements_are_orthogonal} Suppose that $e_1,\dotsc,e_l\geq 0$, $e^\prime_1,\dotsc,e_l^\prime\geq 0$,  $g_{l+1},\dotsc,g_n,\in\ZZ$, and $g_{l+1}^\prime,\dotsc,g_n^\prime\in\ZZ$ are such that
        \[
        (e_1,\dotsc,e_l,g_{l+1},\dotsc,g_n)\neq(e_1^\prime,\dotsc,e_l^\prime,g_{l+1}^\prime,\dotsc,g_n^\prime).
        \]
        Then
		\[
     [\iso_1^{e_1}\dotsm \iso_l^{e_l}\uni_{l+1}^{g_{l+1}}\dotsm \uni_{n}^{g_n} \iso_1^{\ast k_1}\dotsm \iso_l^{\ast k_l}\otimes \ee_{k_1,\dotsc, k_l}]
     \]
      and
      \[
      [\iso_1^{e_1^\prime}\dotsm \iso_l^{e_l^\prime}\uni_{l+1}^{g_{l+1}^\prime}\dotsm \uni_n^{g_n^\prime} \iso_1^{\ast k_1}\dotsm \iso_l^{\ast k_l}\otimes \ee_{k_1,\dotsc, k_l}]
      \]
      are orthogonal in $H_\phi$ for all $k_1,\dotsc,k_l\geq 0$.
\end{lemma}

\begin{proof}
        It is clear that from the definition of $\theta$ in \cref{eq:theta} that
        \[
        \theta(\iso_l^{k_l}\dotsm \iso_1^{k_1}\uni_n^{\ast g_n^\prime}\dotsm \uni_{l+1}^{\ast g_{l+1}^\prime}\iso_l^{\ast e_l^\prime}\dotsm \iso_1^{\ast e_1^\prime}\cdot \iso_1^{e_1}\dotsm \iso_l^{e_l}\uni_{l+1}^{g_{l+1}}\dotsm \uni_n^{g_n} \iso_1^{\ast k_1}\dotsm \iso_l^{\ast k_l}) =0
        \]
        because at least one of the integrals over the one-dimensional tori equals zero. The claimed orthogonality follows from this.
        \end{proof}

We can now prove the desired generalisation of \cref{res:action_equivalent_to_standard_representation}.

\begin{proposition}\label{res:unitary_equivalence_of_summands} Suppose that $k_1,\dotsc,k_l\geq 0$.
        \begin{enumerate}
                \item The elements
                \[
                [\iso_1^{e_1}\dotsm \iso_l^{e_l}\uni_{l+1}^{g_{l+1}}\dotsm \uni_{n}^{g_n} \iso_1^{\ast k_1}\dotsm \iso_l^{\ast k_l}\otimes \ee _{k_1,\dotsc,k_l}]
                \]
                for $e_1,\dotsc,e_l\geq 0$ and $g_{l+1},\dotsc,g_n\in\ZZ$ form an orthonormal basis of $L_{k_1,\dotsc, k_l}$.
                \item The bijection between orthonormal bases of $L_{0,\dotsc, 0}$ and $L_{k_1,\dotsc, k_l}$, sending
                \[
                [\iso_1^{e_1}\dotsm \iso_l^{e_l}\uni_{l+1}^{g_{l+1}}\dotsm \uni_{n}^{g_n} \otimes \ee_{0,\dotsc, 0}]
                \]
                to
                \[
                [\iso_1^{e_1}\dotsm \iso_l^{e_l}\uni_{l+1}^{g_{l+1}}\dotsm \uni_{n}^{g_n} \iso_1^{\ast k_1}\dotsm \iso_l^{\ast k_l}\otimes \ee_{k_1,\dotsc, k_l}]
                \]
                for $e_1,\dotsc,e_l\geq 0$ and $g_{l+1},\dotsc,g_n\in\ZZ$, yields an isomorphism between $L_{0,\dotsc, 0}$ and $L_{k_1,\dotsc, k_l}$ that is a unitary equivalence between the subrepresentations of $\pi_\phi$ of $\UAl$ on $L_{0,\dotsc, 0}$ and on $L_{k_1,\dotsc, k_l}$.
        \end{enumerate}
\end{proposition}

\begin{proof}
        For the first part, we note that \cref{res:if_nonzero_then_norm_one}, combined with \cref{res:algebraic_results}, implies that the elements of the proposed basis have norm one, and that \cref{res:elements_are_orthogonal} shows that they are pairwise orthogonal.

        Furthermore, $L_{k_1,\dotsc,k_n}$ is, by definition, the closure of the linear span of the elements
        \[
        [\iso_1^{e_1}\iso_1^{\ast f_1}\dotsm \iso_l^{e_l}\iso_l^{\ast f_l}\uni_{l+1}^{g_{l+1}}\dotsm \uni_{n}^{g_n} \otimes \ee_{k_1,\dotsc, k_l}]
        \]
        for $e_1,\dotsc,e_l\geq 0$, $f_1,\dotsc,f_l\geq 0$, and $g_{l+1},\dotsc,g_n\in\ZZ$. By \cref{res:if_nonzero_then_norm_one}, this linear span equals that of the same elements
        \[
        [\iso_1^{e_1}\iso_1^{\ast f_1}\dotsm \iso_l^{e_l}\iso_l^{\ast f_l}\uni_{l+1}^{g_{l+1}}\dotsm \uni_{n}^{g_n} \otimes \ee_{k_1,\dotsc, k_l}]
        \]
        for the smaller set of indices
        $e_1,\dotsc,e_l\geq 0$, $f_1,\dotsc,f_l\geq 0$, and $g_{l+1},\dotsc,g_n\in\ZZ$ such that $0\leq f_1\leq k_1,\dotsc,0\leq f_l\leq k_l$. \cref{res:algebraic_results} then shows that this linear span is also the linear span of the elements
        \[
        [\iso_1^{e_1}\dotsm \iso_l^{e_l}\uni_{l+1}^{g_{l+1}}\dotsm \uni_{n}^{g_n} \iso_1^{\ast f_1}\dotsm \iso_l^{\ast f_l}\otimes \ee_{k_1,\dotsc, k_l}]
        \]
        for the same set of indices
        $e_1,\dotsc,e_l\geq 0$, $f_1,\dotsc,f_l\geq 0$, and $g_{l+1},\dotsc,g_n\in\ZZ$ such that $0\leq f_1\leq k_1,\dotsc,0\leq f_l\leq k_l$. We know from
        \cref{res:existence_of_unimodular_constant} that, for some unimodular constant $c$ that depends on the indices,
        \[
        [\iso_1^{\ast f_1}\dotsm \iso_l^{\ast f_l}\otimes \ee_{k_1,\dotsc, k_l}]=c[\iso_1^{k_1-f_1}\dotsm \iso_l^{k_l-f_l} \iso_1^{\ast k_1}\dotsm \iso_l^{\ast k_l}\otimes \ee_{k_1,\dotsc, k_l}].
        \]
        Combining this with \cref{res:algebraic_results}, we see that the linear span under consideration is also the linear span of the elements
        \[
        [\iso_1^{e_1+k_1-f_1}\dotsm \iso_l^{e_l+k_l-f_l}\uni_{l+1}^{g_{l+1}}\dotsm \uni_{n}^{g_n}\iso_1^{\ast k_1}\dotsm \iso_l^{\ast k_l}\otimes\ee_{k_1,\dotsc,k_l}]
        \]
        for still the same set of indices  $e_1,\dotsc,e_l\geq 0$, $f_1,\dotsc,f_l\geq 0$, and $g_{l+1},\dotsc,g_n\in\ZZ$ such that $0\leq f_1\leq k_1,\dotsc,0\leq f_l\leq k_l$. Since these are precisely the elements of the proposed basis (but with multiple occurrences), this concludes the proof of the first part.

        The second part is then obvious.
\end{proof}

Finally, we arrive at the fundamental result of this paper.

\begin{theorem}\label{res:standard_representation_is_injective}
        The standard representation $\strep$ of $\UAl$ on $\Hl$ is injective.
    \end{theorem}

\begin{proof}
	Suppose that $x\in\UAl$ is such that $\strep(x)=0$. Then \cref{res:action_equivalent_to_standard_representation} implies that $\pi_\phi(x)$ acts as zero on $L_{0,\dotsc, 0}$. By \cref{res:unitary_equivalence_of_summands}, $\pi_\phi(x)$ acts as zero on $L_{k_1,\dotsc, k_l}$ for all $k_1,\dotsc,k_l\geq 0$, and then
	\cref{res:H_is_direct_sum} shows that $\pi_\phi(x)=0$. Since $\pi_\phi$ is known to be injective as a consequence of the faithfulness of $\phi$, we conclude that $x=0$.
\end{proof}

\begin{remark}\label{rem:special_cases} In view of \cref{eq:general_isometries}, the injectivity of the standard representation has the following two special cases:
	\begin{enumerate}
		\item  The universal \Calgebra\ that is generated by a unitary is the $\Cstar$-subalge\-bra of $\bounded(\ell^2(\ZZ))$ that is generated by the bilateral shift on $\ell^2(\ZZ)$. Since the spectrum of the bilateral shift is $\TT$, the continuous functional calculus then yields the well-known fact that this universal \Calgebra\ is (isomorphic to) $\cont(\TT)$. Naturally, this same calculus  \emph{immediately} shows that $\cont(\TT)$ is the universal \Calgebra\ that is generated by a unitary, but in the light of the general picture in \cref{res:standard_representation_is_injective} this is not the `natural' description. As the proof of \cref{res:standard_representation_is_injective} shows, however, understanding what the `natural' answer is takes quite some effort.
		\item The universal $\Cstar$-subalgebra that is generated by an isometry is the \Cstar-subalgebra of $\bounded(\ell^2(\NN_0))$ that is generated by the unilateral shift, i.e., the Toeplitz algebra. The proofs in the literature of this well-known result (due to Coburn; see \cite{coburn:1967}) that we are aware of all use the Wold decomposition for one isometry. Although it was the simultaneous Wold decomposition that led to the definition of the standard representation in \cref{eq:general_isometries}, the decomposition itself is \emph{not} used in the present paper.
	\end{enumerate}
\end{remark}

\section{Embeddings between universal algebras}\label{sec:embeddings}


\noindent 
In the previous sections, it has become clear that the standard representation is (most) easily described using the words $s_1^{e_1}\dotsm s_l^{e_l}s_{l+1}^{g_{l+1}} \dotsm s_n^{g_n}$ in $\IUAl$. The action of a generator $s_i$ in this representation is found by `moving $s_i$ to its proper spot' in all products $s_i\cdot s_1^{e_1}\dotsm s_l^{e_l} s_{l+1}^{g_{l+1}}\dotsm s_n^{g_n}$; see \cref{rem:mnemonic}. For each enumeration $e_1,\dotsc,e_l,g_{l+1},\dotsc,g_n$ one can also choose and fix \emph{any} monomial containing precisely that many factors of the $s_i$, and then (attempt to) define a representation of $\UAl$ in an analogous manner. We shall now show that this \emph{does} yield a representation of $\UAl$, and that this representation is always unitarily equivalent to the standard representation. This fact will then help us to show that the natural unital $\ast$-homomorphisms between various algebras $\UAl$ are, in, fact, injective; see \cref{res:algebras_embed}. The details are as follows.

For each $n$-tuple $(e_1,\dotsc,e_l,g_{l+1},\dotsc,g_n)\in\NN_0^l\times \ZZ^{n-l}$, choose and fix a monomial $m_{e_1,\dotsc,e_l,g_{l+1},\dotsc,g_n}\in\IUAl$ which is a product, in any order, of $e_i$ factors $s_i$ and $g_j$ factors $u_j$ ($i=1,\dotsc,l$; $j=l+1,\dotsc,n$).
If $g_j<0$, this is to be read as $|g_j|$ factors $u_j^{-1}$.
We know from \cref{res:basis_and_injectivity} that the monomials thus obtained are linearly independent. Consequently, there are uniquely determined unimodular constants $\lambda(s_i, m_{e_1,\dotsc,e_l,g_{l+1},\dotsc,g_n})$ such that
\begin{equation}\label{eq:unimodular_constants}
s_i\cdot  m_{e_1,\dotsc,e_l,g_{l+1},\dotsc,g_n} = \lambda(s_i, m_{e_1,\dotsc,e_l,g_{l+1},\dotsc,g_n})\ m_{e_1,\dotsc,e_i+1,\dotsc,e_l,g_{l+1},\dotsc,g_n}
\end{equation}
for $i=1,\dotsc,l$, and likewise for $i=l+1,\dotsc,n$.

Let $H$ be a Hilbert space with orthonormal basis $\{b_{m_{e_1,\dotsc,e_l,g_{l+1},\dotsc,g_n}}: e_1,\dotsc,e_l\geq 0; \ g_{l+1},\dotsc,g_{n}\in\ZZ\}$ indexed by the chosen monomials.
For $i=1,\dotsc,l$, we use an anticipating notation to define an isometry $\pi(s_i)$ on $H$ by setting
\begin{equation}\label{eq:representation_via_monomials}
\pi(s_i)\,b_{m_{e_1,\dotsc,e_l,g_{l+1},\dotsc,g_n}}\coloneqq \lambda(s_i,m_{e_1,\dotsc,e_l,g_{l+1},\dotsc,g_n}) b_{m_{e_1,\dotsc,e_i+1,\dotsc,e_l,g_{l+1},\dotsc,g_n}}
\end{equation}
for $i=1,\dotsc,l$, and likewise for $i=l+1,\dotsc,n$.

Let $\{m_{e_1,\dotsc,e_l,g_{l+1},\dotsc,g_n}^\prime: e_1,\dotsc,e_l\geq 0;\ g_{l+1},\dotsc,g_n\in\ZZ\}$ denote a possibly different choice for the  monomials, yielding isometries $\pi^\prime(s_i)$ on a Hilbert space $H^\prime$ with orthonormal basis $\{b_{m_{e_1,\dotsc,e_l,g_{l+1},\dotsc,g_n}^\prime}: e_1,\dotsc,e_l\geq 0;\ g_{l+1},\dotsc,g_n\in\ZZ\}$. We proceed to show that there is a simultaneous unitary equivalence between $\pi(s_i)$ and $\pi^\prime(s_i)$ for $i=1,\dotsc,n$.

There are unimodular constants $\lambda^\prime(s_i, m_{e_1,\dotsc,e_l,g_{l+1},\dotsc,g_n}^\prime)$ such that
\begin{equation}\label{eq:structure_constants}
	s_i\cdot  m_{e_1,\dotsc,e_l,g_{l+1},\dotsc,g_n}^\prime = \lambda^\prime(s_i, m_{e_1,\dotsc,e_l,g_{l+1},\dotsc,g_n}^\prime)\ m_{e_1,\dotsc,e_i+1,\dotsc,e_l,g_{l+1},\dotsc,g_n}^\prime
\end{equation}
for $i=1,\dotsc,l$, and likewise for $i=l+1,\dotsc,n$.
Furthermore, there are unimodular constants $\mu(m_{e_1,\dotsc,e_l,g_{l+1},\dotsc,g_n}^\prime,m_{e_1,\dotsc,e_l,g_{l+1},\dotsc,g_n})$
such that
\begin{equation}\label{eq:unimodular_constants_again}
	m_{e_1,\dotsc,e_l,g_{l+1},\dotsc,g_n}^\prime
=\mu(m_{e_1,\dotsc,e_l,g_{l+1},\dotsc,g_n}^\prime,m_{e_1,\dotsc,e_l,g_{l+1},\dotsc,g_n})\ m_{e_1,\dotsc,e_l,g_{l+1},\dotsc,g_n}.
\end{equation}
Inserting \cref{eq:unimodular_constants_again} into the left hand side of \cref{eq:structure_constants}, and using \cref{eq:unimodular_constants}, we find, for $i=1,\cdots,l$, that
\begin{align*}
&\lambda(s_i, m_{e_1,\dotsc,e_l,g_{l+1},\dotsc,g_n})\ \mu(m_{e_1,\dotsc,e_l,g_{l+1},\dotsc,g_n}^\prime,m_{e_1,\dotsc,e_l,g_{l+1},\dotsc,g_n})\\
&\phantom{=}\,\cdot  \ m_{e_1,\dotsc,e_i+1,\dotsc,e_l,g_{l+1},\dotsc,g_n}\\
&=\lambda^\prime(s_i, m_{e_1,\dotsc,e_l,g_{l+1},\dotsc,g_n}^\prime)\ m_{e_1,\dotsc,e_i+1,\dotsc,e_l,g_{l+1},\dotsc,g_n}^\prime\\
&=\lambda^\prime(s_i, m_{e_1,\dotsc,e_l,g_{l+1},\dotsc,g_n}^\prime)\ \mu (m_{e_1,\dotsc,e_i+1,\dotsc,e_l,g_{l+1},\dotsc,g_n}^\prime,m_{e_1,\dotsc,e_i+1,\dotsc,e_l,g_{l+1},\dotsc,g_n})\\
&\phantom{=}\,\cdot m_{e_1,\dotsc,e_i+1,\dotsc,e_l,g_{l+1},\dotsc,g_n},
\end{align*}
where \cref{eq:unimodular_constants_again} was used in the last step.
Hence
\begin{align*}	&\lambda(s_i, m_{e_1,\dotsc,e_l,g_{l+1},\dotsc,g_n})\ \mu(m_{e_1,\dotsc,e_l,g_{l+1},\dotsc,g_n}^\prime,m_{e_1,\dotsc,e_l,g_{l+1},\dotsc,g_n})
\\
& =\lambda^\prime(s_i, m_{e_1,\dotsc,e_l,g_{l+1},\dotsc,g_n}^\prime)\ \mu (m_{e_1,\dotsc,e_i+1,\dotsc,e_l,g_{l+1},\dotsc,g_n}^\prime,m_{e_1,\dotsc,e_i+1,\dotsc,e_l,g_{l+1},\dotsc,g_n})
\end{align*}
for $i=1,\dotsc,l$, and likewise for $i=l+1,\dotsc,n$.

Using this equality, it is then easily seen that the isomorphism between $H^\prime$ and $H$ such that
\[ b_{m_{e_1,\dotsc,e_l,g_{l+1},\dotsc,g_n}^\prime}^\prime \mapsto\mu(m_{e_1,\dotsc,e_l,g_{l+1},\dotsc,g_n}^\prime,m_{e_1,\dotsc,e_l,g_{l+1},\dotsc,g_n})\ b_{m_{e_1,\dotsc,e_l,g_{l+1},\dotsc,g_n}}
\]
is a unitary equivalence between $\pi^\prime(s_i)$ and $\pi(s_i)$ for $i=1,\dotsc,n$.

For our standard choice $m_{e_1,\dotsc,e_l,g_{l+1},\dotsc,g_n}=s_1^{e_1}\dotsm s_l^{e_l}s_{l+1}^{g_{l+1}} \dotsm s_n^{g_n}$ for the monomials, the above definitions of the $\pi(s_i)$ are known to give the standard representation of $\UAl$ (see \cref{rem:mnemonic}). In view of the existing unitary equivalence, we now see that every choice for the monomials gives a representation of $\UAl$ via \cref{eq:representation_via_monomials} that is unitarily equivalent to the standard representation.

With this available, we can now prove the following theorem, showing that the natural unital $\ast$-homomorphisms\textemdash when they exist\textemdash between the universal \Calgebras\ in this paper are embeddings.

\begin{theorem}\label{res:algebras_embed}
Suppose that $n\geq 0$, and that $z_{ij}\in\TT$ for $i,j=1,\dots,n$ are such that $z_{ji}=\overline{z}_{ij}$ for all $i,j=1,\dotsc,n$ with $i\neq j$. Suppose that $0\leq l\leq n$, and let ${\mathcal A}$ denote the universal unital \Calgebra\ that is generated by isometries $s_1,\dotsc,s_l$ and unitaries $s_{l+1},\dotsc,s_n$ such that $s_i^\ast s_j=\overline{z}_{ij}s_j s_i^\ast$ for $i,j=1,\dotsc,n$ with $i\neq j$.

Suppose that $n^\prime\geq n$, and that $z_{ij}^\prime\in\TT$ for $i,j=1,\dots,n^\prime$ are such that $z_{ji}^\prime=\overline{z\sp\prime}_{\!ij}$ for all $i,j=1,\dotsc,n^\prime$ with $i\neq j$ and $z^\prime_{ij}=z_{ij}$ for all $i,j=1,\dotsc,n$ with $i\neq j$. Suppose that $n\leq k\leq n^\prime$, and let ${\mathcal A}^\prime$ denote the universal unital \Calgebra\ that is generated by isometries $s_1^\prime,\dots,s_l^\prime$, unitaries $s_{l+1}^\prime,\dots,s_n^\prime$, isometries $s_{n+1}^\prime,\dotsc s_k^\prime$, and unitaries $s_{k+1}^\prime,\dotsc,s_{n^\prime}^\prime$ such that $s_i^{\prime\ast} s_j^\prime=\overline{z}_{ij}s_j s_i^{\prime \ast}$ for $i,j=1,\dotsc,n^\prime$ with $i\neq j$.

Then the unique unital $\ast$-homomorphism $\psi:{\mathcal A}\to{\mathcal A}^\prime$ such that $\psi(s_i)=s_i^\prime$ for $i=1,\dotsc,n$ is injective.
\end{theorem}

\begin{proof}
We us the result and notation from the discussion preceding the theorem for the standard representations of both ${\mathcal A}^\prime$ and ${\mathcal A}$. We start with ${\mathcal A}^\prime$, where we have the freedom to choose monomials as we see fit, provided that all necessary combinations of numbers of factors are covered. We make a choice where the generators $s_1^\prime,\dotsc,s_n^\prime$ that `correspond to ${\mathcal A}$' always come first. That is, we make a choice such that the chosen monomials for ${\mathcal A}^\prime$ are those of the form $m_{\mathcal A}^\prime m^\prime$, where $m_{\mathcal A}^\prime$ runs through a fixed choice of monomials as required for $\mathcal A$ (but then in the $s_1^\prime,\dotsc,s_n^\prime)$, and $m^\prime$ runs through a fixed choice for the required monomials in $s_{n+1}^\prime,\dotsc,s_{n^\prime}^\prime$. The standard representation $\pi^\prime$ of ${\mathcal A}^\prime$ can then be realised on a Hilbert space $H^\prime$ with an orthonormal basis consisting of elements $b^\prime_{m_{\mathcal A}^\prime m^\prime}$ that are indexed by the monomials $m_{\mathcal A}^\prime m^\prime$ from our choice. Fix any such monomial $m^\prime$ in $s_{n+1}^\prime,\dotsc,s_{n^\prime}^\prime$ from our choice. A moment's thought shows that the closed linear span $L_{m^\prime}^\prime$ in $H^\prime$ of the basis vectors $b_{m_{\mathcal A}^\prime m^\prime}^\prime$, where $m_{\mathcal A}^\prime$ runs through our fixed choice of monomials as required for $\mathcal A$ (but then in the $s_1^\prime,\dotsc,s_n^\prime)$, is a reducing subspace for $\pi^\prime(s_1^\prime),\dotsc,\pi(s_n^\prime)^\prime$. Moreover, using the material preceding the theorem for the standard representation of $\mathcal A$, it is immediate that the representation $\pi^\prime\circ\psi$ of $\mathcal A$ on $L_{m^\prime}^\prime$ is unitarily equivalent to the standard representation of $\mathcal A$. Since the latter is injective, so is $\psi$.
\end{proof}

\section{Rieffel deformation}\label{sec:deformation}


\noindent
In this section, we show that $\UAl$ is isomorphic to a Rieffel deformation of $\mathcal{A}_{\{1\};l,n-l}$. This implies that $\UAl$ is nuclear, which is a particular case of \cite[Theorem~ 6.2]{rakshit_sarkar_suryawanshi:2022}. It also allows us to compute its K-theory, thus arriving at the results in \cite{bhatt_saurabh_UNPUBLISHED:2023}. At least where the K-theory is concerned, this approach is computationally less demanding than the one in \cite{bhatt_saurabh_UNPUBLISHED:2023}.

We start with a brief review. 
Let $A$ be a \Calgebra, supplied with a strongly continuous action $\alpha$ of $\RR^n$, and denote by $A^\infty$ the set of $a\in A$ such that $x \mapsto \alpha_x(a)$ is a $\mathrm{C}^\infty$\!-function. It is a dense $\ast$-subalgebra of $A$ which is invariant under $\alpha$. Let $\Theta$ be a real skew-symmetric $n\times n$ matrix.  In Rieffel's deformation theory (see \cite{rieffel:1993a}), using oscillatory integrals, one introduces a product $(a,b)\mapsto a\cdot_\Theta b$ on $A^\infty$
such that, with the original involution, $A^\infty$ is a $\ast$-algebra. The $\ast$-algebra $(A^\infty, \cdot_\Theta)$ admits a $\mathrm{C}^\ast$-completion $A_\Theta$ in a ${\mathrm C}^\ast$-norm $\norm{\,\cdot\,}_\Theta$, defined by Hilbert module techniques, such that the action on $A^\infty$ extends to a strongly continuous action of $\RR^n$ on the ${\mathrm C}^\ast$-algebra $A_\Theta$ with $A_\Theta^\infty=A^\infty$.

In our case, we are interested in a periodic action of $\RR^n$, i.e., an action $\alpha$ of $\TT^n$. We identify $\widehat{\TT}^n$ and  $\ZZ^n$ by letting $p=(p_1,...,p_n)\in \ZZ^n$ 
correspond to the character $\chi_p\colon\TT^n\to\TT$, defined by setting $\chi_p(t_1,...,t_n)\coloneqq t_1^{p_1}...t_n^{p_n}$ for $(t_1,\dotsc,t_n)\in\TT^n$. We let $e_1,\dotsc,e_n$ denote the canonical generators of $\ZZ^n$. For $p\in\ZZ^n$, set
\[ 
A_p = \{ a \in A: \alpha_t(a) = \chi_p(t)a \text{ for } t \in \TT^n \} .
\]
Clearly, $A_p\subset A^\infty$. As $A^\infty=A^\infty_\Theta$, we have $A_p=(A_\Theta)_p$ for $p\in\ZZ^n$.  Since $A_p A_q \subset A_{p+q}$ and $A_p^\ast \subset A_{-p}$ for $p, q \in \ZZ^n$, the spaces $A_p$ for  $p \in \ZZ^n$ are the homogeneous components of a
$\ZZ^n$-grading on the $\ast$-subalgebra of $A$ that they span. The general theory of representations of compact groups on Banach spaces implies that
\begin{equation}\label{eq:dense_span}
A = \overline{\bigoplus_{p \in \ZZ^n} A_p}.
\end{equation}

From \cite[Proposition 2.22]{rieffel:1993a}, we have the following explicit formula for the deformed product of homogeneous elements. In it, $\langle\,\cdot\,,\,\cdot\,\rangle$ denotes the usual inner product on $\RR^n$.

\begin{proposition}\label{prop:rieffelprop2.22}
Let $A$ be a $C^\ast$-algebra with a $\TT^n$-action. For $a \in A_p$ and $b \in A_q$, we have
\begin{equation*}
a\cdot_\Theta b = e^{2 \pi i \langle \Theta(p), q \rangle}\,ab.
\end{equation*}
\end{proposition}

We can conclude from this that $\norm{a}_{\Theta}=\norm{a}$ for $p\in\ZZ^n$ and $a\in A_p$. For $p=0$ this is clear from \cite[p.35]{rieffel:1993a} since (in the notation of \cite{rieffel:1993a}) $L_a^J=L_a^0$ for such $a$. For general $p$ and $a\in A_p$, we then have
\[
\norm{a}_\Theta^2=\norm{a\cdot_\Theta a^\ast}_\Theta=\norm{aa^\ast}_\Theta=\norm{aa^\ast}=\norm{a}^2.
\]

Furthermore, combining \cite[Theorem 2.8]{bucholz_lechner_summers:2011} for actions of $\RR^n$ which are periodic in each variable with \cref{prop:rieffelprop2.22}, we have the following. The final statement on injectivity is not needed in the present paper. 

\begin{proposition}\label{res:representation_of_deformation}
Let $A$ be a $C^\ast$-algebra with a $\TT^n$-action $\alpha$, let $\rho$ be a unitary representation of $\TT^n$ on a Hilbert space $H$, and let $\pi$ be a representation of $A$ on $H$ such that 
\[
\pi(\alpha_t(a))=\rho_t \pi(a)\rho_t^\ast
\]
for $a\in A$ and $t\in \TT^n$. For $p\in\ZZ^n$, set

\[
H_p\coloneqq \{\xi\in H : \rho(t)\xi=\chi_p(t)\xi\text{ for }t\in\TT^n\}.
\]
There exists a unique  representation $\pi_\Theta$ of $A_\Theta$ on $H$ such that 
\[
\pi_\Theta(a) \xi = e^{2 \pi i \langle \Theta(p), q \rangle} \pi(a) \xi
\]
for $\xi \in H_q$, $a \in A_p$, and $p,q\in \ZZ^n$.
Moreover, $\pi_\Theta$ is injective if and only if $\pi$ is injective.
\end{proposition}

After these preparations, we now set out to prove that $\UAl$ is isomorphic to a Rieffel deformation of $\mathcal{A}_{\{1\};l,n-l}$. In order to keep our notation straight, we denote the generating isometries of the former algebra by the usual $s_1,\dotsc,s_n$ and those of the latter by $s_1^\prime,\dotsc,s_n^\prime$. Likewise, we let $\strep$ denote the standard representation of $\UAl$ on $H_{l,n-l}$ and $\strep^\prime$ the standard representation of $\mathcal{A}_{\{1\};l,n-l}$ on the same Hilbert space.

There is an action $\alpha$ of $\TT^n$ on $\mathcal{A}_{\{1\};l,n-l}$ which is determined by
\[
\alpha_{(t_1,\dotsc,t_n)}( s_i^\prime )= t_i s_i^\prime
\]
for $i=1,\dots,n$. For $i,j$ with $1\leq i<j\leq n$, choose $\varphi_{ij}\in\RR$ such that $z_{ij} = e^{2\pi {\mathrm i} \varphi_{ij}}$, and set
\begin{equation}\label{theta_q}
\Theta  \coloneqq \left( \begin{array}{cccc}
    0 & -\frac{\varphi_{12}}{2} & \dotsc & -\frac{\varphi_{1n}}{2}  \\
    \frac{\varphi_{12}}{2} & 0 & \dotsc & -\frac{\varphi_{2n}}{2} \\
    \vdots &  & \ddots & \\
    \frac{\varphi_{1n}}{2} & \frac{\varphi_{2n}}{2} & \dotsc & 0
\end{array} \right).
\end{equation}
Explicitly, set
\[
\Theta_{ij}\coloneqq
\begin{cases}
	\phantom{-}0&\text{when }i=j;\\
	\phantom{-}\frac{\varphi_{ji}}{2} &\text{when }i>j;\\
	-\frac{\varphi_{ij}}{2} &\text{when }i<j.
	\end{cases}
\]

\begin{proposition}\label{res:surjective_morphism}
There exists a unique unital $\ast$-homomorphism $\Psi: \UAl \rightarrow (\mathcal{A}_{\{1\}; l,n-l})_\Theta$ such that
\[ 
\Psi(s_i) = s_i^\prime
\]
for $i = 1,\dotsc, n$. It is surjective.
\end{proposition}
\begin{proof}
Note that $\Psi(s_i)$ belongs to the $e_i$-homogeneous component of $\mathcal{A}_{\{1\};l,n-l}$. Since $\langle \Theta (p), p \rangle = 0$ for any skew-symmetric matrix $\Theta$ and any $p \in \ZZ^n$, one has from \cref{prop:rieffelprop2.22} that 
\[
\Psi(s_i)^\ast\cdot_\Theta \Psi(s_i)=(s_i^\prime)^\ast s_i^\prime =  1
\]
for $i = 1,\dotsc, n$, and 
\[
\Psi(s_i)\cdot_\Theta \Psi(s_i)^\ast= s_i^\prime (s_i^\prime)^\ast =  1
\]
for $i = l + 1,\dotsc, n$. Hence $\Psi(s_i)$ is an isometry for $i=1,...,l$ and a unitary for $i=l+1,...,n$.

Furthermore, for $k,r$ with $1\leq k<r\leq n$, we have:
\begin{equation*}
    \begin{split}
        \Psi(s_k)^\ast \cdot_\Theta \Psi(s_r)&=e^{2\pi {\mathrm i}\langle \Theta (-e_k) \, ,\, e_r\rangle}
(s_k^\prime)^\ast s_r^\prime \\
 & = e^{-\pi {\mathrm i}\varphi_{kr} 
 }(s_k^\prime)^\ast s_r^\prime   \\ 
 &
=\overline{z}_{kr}\left(z_{kr}e^{-\pi {\mathrm i}\varphi_{kr}}\right)s_r^\prime (s_k^\prime)^\ast\\ 
& =
\overline{z}_{kr}\left(e^{2\pi {\mathrm i}\varphi_{kr}}e^{-\pi {\mathrm i}\varphi_{kr}}\right)s_r^\prime (s_k^\prime)^\ast\\ 
&
=\overline{z}_{kr}\left(e^{2\pi {\mathrm i}\langle \Theta(e_r)\, ,\, -e_k\rangle} s_r^\prime (s_k^\prime)^\ast\right)\\
&=\overline{z}_{kr}\Psi(s_r) \cdot_\Theta \Psi(s_k)^\ast.
    \end{split}
\end{equation*}
Taking adjoints, it follows from this that also $\Psi(s_k)^\ast \cdot_\Theta \Psi(s_r)=\overline{z}_{kr}\Psi(s_r) \cdot_\Theta \Psi(s_k)^\ast$ for $k,r$ with $1\leq r<k\leq n$. By the universal property of $\UAl$, the existence and uniqueness of $\Psi$ are now clear.

It remains to be shown that $\Psi$ is surjective. In view of \cref{eq:dense_span},  applied to  $(\mathcal{A}_{\{1\}; l,n-l})_\Theta$, it is sufficient to show that $((\mathcal{A}_{\{1\}; l,n-l})_\Theta)_p$ is contained in the image of $\Psi$ for $p\in\ZZ^n$. Take $p\in\ZZ^n$, $a\in((\mathcal{A}_{\{1\}; l,n-l})_\Theta)_p=(\mathcal{A}_{\{1\}; l,n-l})_p$, and a sequence $(a_k)$ in $\mathcal{A}^0_{\{1\}; l,n-l}$ such that $a_k\to a$ in the norm of $\mathcal{A}_{\{1\}; l,n-l}$. There exists a projection $P_p\colon \mathcal{A}_{\{1\}; l,n-l}\to (\mathcal{A}_{\{1\}; l,n-l})_p$ such that  
\[
P_p(a)=\int_{\TT^n}\overline{\chi_p}(t)\alpha_t(a)\,{\mathrm d}t
\]
for $a\in  \mathcal{A}_{\{1\}; l,n-l}$.
Hence $P_p(a_k)\to a$ in the norm of $\mathcal{A}_{\{1\}; l,n-l}$. As this convergence is in $(\mathcal{A}_{\{1\}; l,n-l})_p=((\mathcal{A}_{\{1\}; l,n-l})_\Theta)_p$, where the original and the deformed norm agree, we see that $P_p(a_k)\to a$ in $(\mathcal{A}_{\{1\}; l,n-l})_\Theta$. Since $P_p(a_k)\in\mathcal{A}^0_{\{1\}; l,n-l}$ which (in view of \cref{prop:rieffelprop2.22}) is clearly contained in the image of $\Psi$, we see that $a$ is also in this image.
\end{proof}

\begin{remark}
	It follows from \cite[Proposition~4.10]{rieffel:1993a} that a subset of $A_p$ which is dense in the original norm is still dense in the deformed norm. This is already sufficient to prove that $\Psi$ is surjective, bypassing the stronger statement that the norm on $A_p$ is actually unchanged.
\end{remark}

The $\ast$-homomorphism $\Psi$ in \cref{res:surjective_morphism} is also injective. We shall now prove this, by relating it to the standard representation $\strep$ of $\UAl$ on $H_{l,n-l}$ (which we know to be injective) via a representation of $(\mathcal{A}_{\{1\}; l,n-l})_\Theta$ on the same Hilbert space that can be obtained from \cref{res:representation_of_deformation}.

As a first preparation for this, define a representation $\rho$ of $\TT^n$ on $H_{l,n-l}$ by setting
\[
\rho_{(t_1,\dotsc,t_n)}\ep_{k_1,\dotsc,k_n}\coloneqq t_1^{k_1}\cdots t_n^{k_n}\ep_{k_1,\dotsc,k_n}.
\]
It is easily verified that $\strep^\prime(\alpha_t(s_i^\prime))=\rho_t\strep^\prime(s_i^\prime)\rho_t^\ast$ for $i=1,\dotsc,n$ and $t\in\TT^n$, implying that $\strep^\prime(\alpha_t(a))=\rho_t\strep^\prime(a)\rho_t^\ast$ for $a\in \mathcal{A}_{\{1\};l,n-l}$ and $t\in\TT^n$.
Hence \cref{res:representation_of_deformation} yields a representation $(\strep^\prime)_\Theta$ of $(\mathcal{A}_{\{1\}; l,n-l})_\Theta$ on $H_{l,n-l}$. 

As a second preparation, we define a unitary operator $T\colon H_{l,n-l}\to H_{l,n-l}$, as follows. For $i,j=1,\dotsc,n$, set
\begin{equation*}
	w_{i,j}\coloneqq
	\begin{cases}
		1&\text{ when }i=j;\\
		e^{\pi{\mathrm i}\varphi_{ij}}& \text{ when }i<j;\\
		e^{-\pi{\mathrm i}\varphi_{ji}} & \text{ when } i>j.
	\end{cases}
\end{equation*}
Then $w_{i,j}=\overline{w}_{j,i}$ for $i,j=1,\dotsc,n$ and $w_{i,j}^2=z_{i,j}$ when $i\neq j$. The diagonal unitary operator $T$ on $H_{l,n-l}$ is then defined by setting
\[
T\ep_{k_1,\dotsc,k_n}\coloneqq \prod_{p>q}w_{p,q}^{k_pk_q}\ep_{k_1,\dotsc,k_n}.
\]

The sought relation between the injective representation $\strep$ and the $\ast$-homo\-morph\-ism $\Psi$ is as follows.

\begin{proposition}\label{res:unitary_equivalence}
For $a\in\UAl$, we have
\begin{equation}\label{eq:relation_between_standardrepresentations}
\left[(\strep^\prime)_\Theta \circ \Psi\right](a) = T^\ast \strep(a) T.
\end{equation}

\end{proposition}
\begin{proof}
It is sufficient to prove \cref{eq:relation_between_standardrepresentations} when $a$ is one of the generators $s_1,\dotsc,s_n$ of $\UAl$.

Note that $\CC\cdot \ep_{k_1 \dotsc k_n}$ is the $(k_1, \dotsc, k_n)$-homogeneous component of $H_{l,n-l}$ for the representation $\rho$ of $\TT^n$ on this space. Using \cref{res:representation_of_deformation}, we then have that

\begin{equation*}
    \begin{split}
    \big[(\strep^\prime)_\Theta \circ  \Psi\big](s_i) &\ep_{k_1, \dotsc, k_n} = \left[(\strep^\prime)_\Theta (s_i^\prime)\right] \ep_{k_1, \dotsc, k_n} \\
     &= e^{2 \pi {\mathrm i}\langle \Theta (e_i)\,,\, (k_1, \dotsc, k_n) \rangle}  \ep_{k_1,\dotsc,k_{i-1}, k_i+1,k_{i+1},\dotsc,k_n}  \\ 
     & = e^{\pi {\mathrm i} (-k_1 \varphi_{1i} -\, \dotsb\, - k_{i-1} \varphi_{i-1,i} + k_{i+1} \varphi_{i,i+1}+\,\dotsb\,+k_n \varphi_{in})} \ep_{k_1,\dotsc,k_{i-1}, k_i+1,k_{i+1},\dotsc,k_n}  \\ 
     & = w_{i,1}^{k_1} \dotsb w_{i,n}^{k_n} \ep_{k_1,\dotsc,k_{i-1}, k_i+1,k_{i+1},\dotsc,k_n}.
    \end{split}
\end{equation*}

Then
\begin{equation*}
    \begin{split}
   T&\big[\big[(\strep^\prime)_\Theta \circ  \Psi\big](s_i) \ep_{k_1, \dotsc, k_n}\big]
  =w_{i,1}^{k_1} \dotsb w_{i,n}^{k_n} T \ep_{k_1,\dotsc,k_{i-1},k_i+1,k_{i+1}\dotsc,k_n}\\
   &= w_{i,1}^{k_1} \dotsb w_{i,n}^{k_n}\prod_{{\begin{scriptsize}\begin{matrix} p>q\\ p,q\not= i\end{matrix}\end{scriptsize}}}w_{p,q}^{k_p k_q} \prod_{q<i} w_{i,q}^{(k_i+1) k_q}
\prod_{p>i}w_{p,i}^{k_p (k_i+1)}   \ep_{k_1,\dotsc,k_{i-1},k_i+1,k_{i+1},\dotsc,k_n}\\ 
&= z_{i,1}^{k_1} \dotsb z_{i,i-1}^{k_{i-1}} \prod_{p > q} w_{p,q}^{k_p k_q}\ep_{k_1,\dotsc,k_{i-1},k_i+1,k_{i+1},\dotsc,k_n}\\
&=
\strep(s_i) (T \ep_{k_1,\dotsc,k_n}).
    \end{split}
\end{equation*}
Here the third equality follows by bringing the first part of  
\[
w_{i,1}^{k_1} \dotsb w_{i,n}^{k_n}=w_{i,1}^{k_1}\dotsb w_{i,i-1}^{k_{i-1}}\,\cdot\, w_{i,i+1}^{k_{i+1}}\dotsb w_{i,n}^{k_n}
\]
into the second product, the second part of it into the third product, and then using that 
\begin{align*}
w_{i,q}^{k_q}  w_{i,q}^{(k_i+1) k_q}&= w_{i,q}^{2k_q+k_ik_q}=z_{i,q}^{k_q} w_{i,q}^{k_i k_r}
\intertext{for $q<i$, and that}
w_{i,p}^{k_p} w_{p,i}^{k_p (k_i+1)}&= w_{p,i}^{-k_p} w_{p,i}^{k_p (k_i+1)}=w_{p,i}^{k_p k_i}
\end{align*}
for $p>i$.

We conclude that
\[
T\, \big[(\strep^\prime)_\Theta \circ  \Psi\big](s_i)=\strep(s_i)\,T.
\]
As $T$ is unitary, this implies the validity of \cref{eq:relation_between_standardrepresentations} for the generators $s_1,\dotsc,s_n$ of $\UAl$, as desired.

\end{proof}
Since $\strep$ is injective by  \cref{res:standard_representation_is_injective}, \cref{res:unitary_equivalence} implies that  $\Psi$ is likewise injective. Combining this with its surjectivity from \cref{res:surjective_morphism} yields the following.

\begin{theorem}\label{corRiefD}
The $C^\ast$-algebras $\UAl$ and $(\mathcal{A}_{\{1\};l,n-l})_\Theta$ are isomorphic via $\Psi$.
\end{theorem}

It is well known (see also \cref{rem:special_cases}) that the universal unital \Calgebra\ generated by one isometry is the Toeplitz algebra $\mathcal T$; for a unitary it is $\cont(\TT)$. It is then easy to see that the unital tensor product $\mathcal{T}^{\otimes l} \otimes \cont(\TT)^{\otimes (n-l)}$ of these nuclear algebras has all the properties required in \cref{def:universal_algebra} when all $z_{ij}$ are equal to 1. Hence it is (isomorphic to) our ${\mathcal A}_{\{1\};l,n-l}$.  

We can now provide an alternate proof of the following special case of \cite[Theorem~6.2]{rakshit_sarkar_suryawanshi:2022}\
\begin{proposition}
The \Calgebra\ ${\UAl}$ is nuclear.
\end{proposition}
\begin{proof}
By \cite{rieffel:1993b}, a deformed \Calgebra\ $A_\Theta$ is nuclear if $A$ is. So the result follows from \cref{corRiefD} and the nuclearity of $\mathcal{T}^{\otimes l} \otimes \cont(\TT)^{\otimes (n-l)}$.
\end{proof}

Using the Pimsner-Voiculescu six term exact sequence associated with a crossed product by $\ZZ$, the K-groups of $\UAl$, together with their generators, are computed in \cite[Theorems~3.5 and~3.6]{bhatt_saurabh_UNPUBLISHED:2023}. Our approach via Rieffel's deformation theory provides an alternate way to compute these groups.

\begin{theorem}\label{res:K-groups}
The K-groups of $\UAl$ are as follows:

$\bullet$\ for $n=l$ with $l\geq 1$, we have
                $K_0(\mathcal A_{\{z_{ij}\}; l,0})\simeq\ZZ$ and $K_1(\mathcal A_{\{z_{ij}\};l,0})\simeq 0$;

$\bullet$\ for $n>l$ with $l\geq 0$,  we have
$K_0(\mathcal A_{\{z_{ij}\}; l,n-l})  \simeq  K_1(\mathcal A_{\{z_{ij}\};l,n-l})\simeq \ZZ^{2^{n-l-1}}$.
\end{theorem}
\begin{proof}
By \cite{rieffel:1993b}, the K-groups of a \Calgebra\ $A$ and its deformation $A_\Theta$ are isomorphic. Hence  \cref{corRiefD} implies that $K_i(\mathcal A_{\{z_{ij}\};l,n-l})\simeq K_i(\mathcal{T}^{\otimes l} \otimes \cont(\TT)^{\otimes (n-l)})$ for $i=0,1$. The K-groups of $\mathcal{T}^{\otimes l}\otimes \cont(\TT)^{\otimes (n-l)}$ are easily calculated. Indeed, on recalling that $K_0(\mathcal{T})\simeq\ZZ$, that $K_1(\mathcal{T})=0$, and that $K_0(\cont(\TT))\simeq\ZZ\simeq K_1(\cont(\TT))$, which are all torsion free groups, a repeated use of the K\"{u}nneth theorem  (see, e.g.,  \cite[p.171]{wegge-olsen_K-THEORY_AND_C-STAR-ALGEBRAS:1993}) readily leads to the results.
\end{proof}





\subsection*{Acknowledgements} The results in this paper were partly obtained during a research visit of the first author to the University of Lisbon. The kind hospitality of the Instituto Superior T\'{e}cnico is gratefully acknowledged. Pinto was partially funded by FCT/Portugal through project UIDB/04459/2020 with DOI identifier 10-54499/UIDP/04459/2020.
The authors thank Lyudmila Turowska for helpful discussions.





\bibliographystyle{amsplain}
\urlstyle{same}
\bibliography{general_bibliography}

\end{document}